\documentclass[10pt, a4paper, fleqn]{article}
\usepackage{titling}
\title{A Finite-Volume Scheme for\protect\\Fractional Diffusion on Bounded Domains}
\newcommand{\titlePDF}{A Finite-Volume Scheme for Fractional Diffusion on Bounded Domains}
\newcommand{\authorPDF}{Bailo, Carrillo, Fronzoni, Gomez-Castro.}
\newcommand{\subjectPDF}{35R11;	65N08.}
\newcommand{\keywordsPDF}{Fractional Laplacian; Levy-Fokker-Planck equation; finite-volume schemes.}
 \usepackage{authblk}

\author[1]{Rafael Bailo}
\author[1]{Jos\'{e} A. Carrillo}
\author[1]{\authorcr Stefano Fronzoni}
\author[2,1]{David G\'{o}mez-Castro}

\affil[1]{
	Mathematical Institute, University of Oxford
}
\affil[ ]{
	OX2 6GG Oxford, United Kingdom
}
\affil[ ]{\textit{bailo@maths.ox.ac.uk, carrillo@maths.ox.ac.uk, fronzoni@maths.ox.ac.uk}}

\affil[ ]{}

\affil[2]{
	Departamento de Matemáticas, Universidad Autónoma de Madrid
}
\affil[ ]{
	Ciudad Universitaria de Cantoblanco, 28049 Madrid, Spain
}
\affil[ ]{\textit{david.gomezcastro@uam.es}}
 
\usepackage[left=1.1in, right=1.1in, top=1.1in, bottom=1.1in]{geometry}

\makeatletter
\let\newtitle\@title
\let\newauthor\@author
\let\newdate\@date
\makeatother
\usepackage{fancyhdr}
\usepackage[
	hypertexnames=false,
	pdftex,
	pdftitle={\titlePDF. \authorPDF},
	pdfauthor={\authorPDF},
	pdfsubject={\subjectPDF},
	pdfkeywords={\keywordsPDF},
]{hyperref}
\hypersetup{
	colorlinks=true,
	linkcolor=blue,
	citecolor=blue,
	urlcolor=blue,
	linktocpage
} 

\usepackage[utf8]{inputenc}
\usepackage[T1]{fontenc}
\usepackage{amsmath}
\usepackage{amsthm}
\usepackage{amsfonts}
\usepackage{amssymb}
\usepackage{graphicx}
\usepackage{mathtools}
\usepackage{ifthen}
\usepackage{dsfont}

\usepackage[UKenglish]{babel}
 
\usepackage[dvipsnames]{xcolor}

\definecolor{ppGreen}{HTML}{008000}
\definecolor{ppBlue}{HTML}{0000FF}
\definecolor{ppRed}{HTML}{FF0000}
\definecolor{ppPurple}{HTML}{800080}
\definecolor{lightblue}{rgb}{0.145,0.6666,1}
\definecolor{grey52}{RGB}{52,52,52}
\definecolor{color1}{RGB}{0,62,116}
\definecolor{color2}{RGB}{152,152,152}
\definecolor{color3}{RGB}{52,52,52}
\definecolor{color4}{RGB}{100,100,100}

\definecolor{imperialnavy}{RGB}{0,33,71}
\definecolor{imperialblue}{RGB}{0,62,116}
\definecolor{imperialgrey}{RGB}{235,238,238}
\definecolor{imperialcoolgrey}{RGB}{157,157,157}

\definecolor{lille}{RGB}{178, 35, 114} 
\newcounter{review}

\newcommand{\ntcreview}[3]{
	\refstepcounter{review}
	{\color{#2}{\textbf{[#1]}: #3}}
}

\newcommand{\creview}[3]{
	\ntcreview{#1}{#2}{#3}
	\addcontentsline{tor}{subsection}{
		\thereview~\textbf{[#1]}:~#3
	}
}

\newcommand{\review}[2]{\creview{#1}{blue}{#2}}

\makeatletter

\newcommand\listreviewname{List of Reviews}
\newcommand\listofreviews{
	\section*{\listreviewname}\@starttoc{tor}}
\makeatother

\usepackage[all]{nowidow}

\newcommand{\subjectclassification}[1]{

	{\small\textbf{\textit{AMS Subject Classification --- }} #1}

}
\newcommand{\keywords}[1]{

	{\small\textbf{\textit{Keywords --- }} #1}

}

\renewcommand\lll\MoveEqLeft

\usepackage{tikz}

\usepackage{pgfplots}
\pgfplotsset{compat=1.16}

\usepackage{float}
\usepackage{subcaption}
\usepackage[section]{placeins}
\usepackage{rotating}

\usepackage{array}
\newcolumntype{L}[1]{>{\raggedright\let\newline\\\arraybackslash\hspace{0pt}}m{#1}}
\newcolumntype{C}[1]{>{\centering\let\newline\\\arraybackslash\hspace{0pt}}m{#1}}
\newcolumntype{R}[1]{>{\raggedleft\let\newline\\\arraybackslash\hspace{0pt}}m{#1}}
\usepackage{booktabs}
\usepackage{tabularx}
\usepackage{multirow}

\usepackage{titling}

\usepackage[algoruled]{algorithm2e}

\usepackage[capitalise]{cleveref}

 \usepackage{accents}

\newcommand\term\emph

\numberwithin{equation}{section}

\usepackage{autobreak}

\usepackage{enumerate}

\makeatletter
\def\@maketitle{
	\newpage
	\begin{center}
		\let \footnote \thanks
		{\LARGE\bfseries \@title \par}
		\vskip 2.5em
			{\large
				\lineskip .5em
				\begin{tabular}[t]{c}
					\@author
				\end{tabular}\par}
		\vskip 1em
			{\large \@date}
	\end{center}
	\par
	\vskip 1.5em}
\makeatother


\usepackage{amsthm}
\theoremstyle{plain}
\newtheorem{theorem}{Theorem}[section]

\theoremstyle{remark}
\newtheorem{remark}[theorem]{\bf Remark}

\newcommand{\defeq}{\coloneqq}

\usepackage{esint}

\def\XXint#1#2#3{{\setbox0=\hbox{$#1{#2#3}{\int}$ }
			\vcenter{\hbox{$#2#3$ }}\kern-.6\wd0}}

\DeclarePairedDelimiter{\prt}{(}{)}
\DeclarePairedDelimiter{\brk}{[}{]}

\DeclarePairedDelimiter{\abs}{|}{|}
\DeclarePairedDelimiter{\norm}{\|}{\|}
\DeclarePairedDelimiter{\set}{\{}{\}}

\DeclarePairedDelimiter{\inn}{\langle}{\rangle}

\usepackage{suffix}

\newcommand{\inner}[2]{\inn{#1,#2}}
\WithSuffix\newcommand\inner*[2]{\inn*{#1,#2}}

\DeclarePairedDelimiter{\positive}{(}{)^{+}}
\DeclarePairedDelimiter{\negative}{(}{)^{-}}

\newcommand\pos\positive
\renewcommand\neg\negative

\WithSuffix\newcommand\pos*{\positive*}
\WithSuffix\newcommand\neg*{\negative*}

\newcommand{\R}{{\mathbb{R}}}

\newcommand{\Rd}{{\mathbb{R}^d}}

\renewcommand{\L}[1]{{L^{#1}}}
\newcommand{\Lone}{\L{1}}
\newcommand{\Ltwo}{\L{2}}

\newcommand{\pnorm}[2]{\norm{#2}_{\L{#1}}}
\WithSuffix\newcommand\pnorm*[2]{\norm*{#2}_{\L{#1}}}

\newcommand{\psnorm}[3]{\norm{#3}_{\L{#1}(#2)}}
\WithSuffix\newcommand\psnorm*[3]{\norm*{#3}_{\L{#1}(#2)}}

\newcommand{\pnormp}[2]{\pnorm{#1}{#2}^{#1}}
\WithSuffix\newcommand\pnormp*[2]{\pnorm*{#1}{#2}^{#1}}

\newcommand{\psnormp}[3]{\psnorm{#1}{#2}{#3}^{#1}}
\WithSuffix\newcommand\psnormp*[3]{\psnorm*{#1}{#2}{#3}^{#1}}

\newcommand\svec\vec

\renewcommand{\vec}{\mathbf}
\renewcommand{\svec}{\boldsymbol}

\renewcommand{\d}{\mathrm{d}}
\newcommand{\dd}{\mathop{}\!\d}

\newcommand{\der}[2]{\frac{\d #1}{\d #2}}
\newcommand{\pder}[2]{\frac{\partial #1}{\partial #2}}

\newcommand{\dt}{\dd t}

\newcommand{\dx}{\dd x}
\newcommand{\dy}{\dd y}
\newcommand{\du}{\dd u}
\newcommand{\dv}{\dd v}

\newcommand{\grad}{\nabla}
\renewcommand{\div}{\nabla\cdot}

\newcommand{\laplacian}{\Delta}
\newcommand{\laplace}{\laplacian}

\newcommand{\pt}{\partial_t}

\newcommand{\Dt}{\Delta t}
\newcommand{\Dx}{\Delta x}
\newcommand{\Dy}{\Delta y}

\newcommand{\nhalf}{\frac{1}{2}}

\renewcommand{\i}{_{i}}
\newcommand{\ip}{_{i+1}}
\newcommand{\im}{_{i-1}}
\newcommand{\ih}{_{i+\nhalf}}
\newcommand{\imh}{_{i-\nhalf}}

\renewcommand{\j}{_{j}}
\newcommand{\jp}{_{j+1}}

\newcommand{\jh}{_{j+\nhalf}}
\newcommand{\jmh}{_{j-\nhalf}}

\renewcommand{\k}{_{k}}

\newcommand{\kh}{_{k+\nhalf}}
\newcommand{\kmh}{_{k-\nhalf}}

\renewcommand{\ij}{_{i,\,j}}
\newcommand{\ipj}{_{i+1,\,j}}

\newcommand{\ihj}{_{i+\nhalf,\,j}}
\newcommand{\imhj}{_{i-\nhalf,\,j}}

\newcommand{\ijp}{_{i,\,j+1}}

\newcommand{\ijh}{_{i,\,j+\nhalf}}
\newcommand{\ijmh}{_{i,\,j-\nhalf}}

\newcommand{\kl}{_{k,\,l}}

\newcommand{\m}{^{m}}
\renewcommand{\mp}{^{m+1}}
\newcommand{\mh}{^{m+\nhalf}}

\newcommand{\ppr}{(r)}

\newcommand{\Wr}{^{W,\,\ppr}}

\newlength{\dhatheight}

\newcommand{\E}{^{E}}

\newcommand{\W}{^{W}}

\ifthenelse{\isundefined{\Wr}}{
	\newcommand{\Wr}{^{W,\,\ppr}}
}{
	\renewcommand{\Wr}{^{W,\,\ppr}}
}

\DeclareMathOperator*{\minmod}{minmod}

\newcommand{\curlyC}{\mathcal{C}}

\newcommand{\ik}{_{i,\,k}}

\newcommand{\bbrho}{\bm{\bar{\rho}}}

 \usepackage{todonotes}

\newif\ifskiptable

\pgfplotsset{
	colormap={hsv}{
			hsb(0.00cm)=(0.00,0,0.95);
			hsb(0.05cm)=(0.05,1,1);
			hsb(0.10cm)=(0.10,1,1);
			hsb(0.15cm)=(0.15,1,1);
			hsb(0.20cm)=(0.20,1,1);
			hsb(0.25cm)=(0.25,1,1);
			hsb(0.30cm)=(0.30,1,1);
			hsb(0.35cm)=(0.35,1,1);
			hsb(0.40cm)=(0.40,1,1);
			hsb(0.45cm)=(0.45,1,1);
			hsb(0.50cm)=(0.50,1,1);
			hsb(0.55cm)=(0.55,1,1);
			hsb(0.60cm)=(0.60,1,1);
			hsb(0.65cm)=(0.65,1,1);
			hsb(0.70cm)=(0.70,1,1);
			hsb(0.75cm)=(0.75,1,1);
			hsb(0.80cm)=(0.80,1,1);
			hsb(0.85cm)=(0.85,1,1);
			hsb(0.90cm)=(0.90,1,1);
			hsb(0.95cm)=(0.95,1,1);
			hsb(1.00cm)=(1.00,1,1);
		}
}

\pgfplotsset{
	colormap={hsvSoft}{
			hsb(0.00cm)=(0.00,0,0.95);
			hsb(0.05cm)=(0.05,1,1);
			hsb(0.10cm)=(0.10,1,1);
			hsb(0.15cm)=(0.15,1,1);
			hsb(0.20cm)=(0.20,1,1);
			hsb(0.25cm)=(0.25,1,1);
			hsb(0.30cm)=(0.30,1,1);
			hsb(0.35cm)=(0.35,1,1);
			hsb(0.40cm)=(0.40,1,1);
			hsb(0.45cm)=(0.45,1,1);
			hsb(0.50cm)=(0.50,1,1);
			hsb(0.55cm)=(0.55,1,1);
			hsb(0.60cm)=(0.60,1,1);
			hsb(0.65cm)=(0.65,1,1);
			hsb(0.70cm)=(0.70,1,1);
			hsb(0.75cm)=(0.75,1,1);
			hsb(0.80cm)=(0.80,1,1);
			hsb(0.85cm)=(0.85,1,1);
			hsb(0.90cm)=(0.90,1,1);
			hsb(0.95cm)=(0.95,1,1);
			hsb(1.00cm)=(0.00,0,0.95);
		}
}

\pgfplotsset{
	colormap={viridisFull}{
			rgb=(0.26700401, 0.00487433, 0.32941519)
			rgb=(0.26851048, 0.00960483, 0.33542652)
			rgb=(0.26994384, 0.01462494, 0.34137895)
			rgb=(0.27130489, 0.01994186, 0.34726862)
			rgb=(0.27259384, 0.02556309, 0.35309303)
			rgb=(0.27380934, 0.03149748, 0.35885256)
			rgb=(0.27495242, 0.03775181, 0.36454323)
			rgb=(0.27602238, 0.04416723, 0.37016418)
			rgb=(0.2770184 , 0.05034437, 0.37571452)
			rgb=(0.27794143, 0.05632444, 0.38119074)
			rgb=(0.27879067, 0.06214536, 0.38659204)
			rgb=(0.2795655 , 0.06783587, 0.39191723)
			rgb=(0.28026658, 0.07341724, 0.39716349)
			rgb=(0.28089358, 0.07890703, 0.40232944)
			rgb=(0.28144581, 0.0843197 , 0.40741404)
			rgb=(0.28192358, 0.08966622, 0.41241521)
			rgb=(0.28232739, 0.09495545, 0.41733086)
			rgb=(0.28265633, 0.10019576, 0.42216032)
			rgb=(0.28291049, 0.10539345, 0.42690202)
			rgb=(0.28309095, 0.11055307, 0.43155375)
			rgb=(0.28319704, 0.11567966, 0.43611482)
			rgb=(0.28322882, 0.12077701, 0.44058404)
			rgb=(0.28318684, 0.12584799, 0.44496 )
			rgb=(0.283072 , 0.13089477, 0.44924127)
			rgb=(0.28288389, 0.13592005, 0.45342734)
			rgb=(0.28262297, 0.14092556, 0.45751726)
			rgb=(0.28229037, 0.14591233, 0.46150995)
			rgb=(0.28188676, 0.15088147, 0.46540474)
			rgb=(0.28141228, 0.15583425, 0.46920128)
			rgb=(0.28086773, 0.16077132, 0.47289909)
			rgb=(0.28025468, 0.16569272, 0.47649762)
			rgb=(0.27957399, 0.17059884, 0.47999675)
			rgb=(0.27882618, 0.1754902 , 0.48339654)
			rgb=(0.27801236, 0.18036684, 0.48669702)
			rgb=(0.27713437, 0.18522836, 0.48989831)
			rgb=(0.27619376, 0.19007447, 0.49300074)
			rgb=(0.27519116, 0.1949054 , 0.49600488)
			rgb=(0.27412802, 0.19972086, 0.49891131)
			rgb=(0.27300596, 0.20452049, 0.50172076)
			rgb=(0.27182812, 0.20930306, 0.50443413)
			rgb=(0.27059473, 0.21406899, 0.50705243)
			rgb=(0.26930756, 0.21881782, 0.50957678)
			rgb=(0.26796846, 0.22354911, 0.5120084 )
			rgb=(0.26657984, 0.2282621 , 0.5143487 )
			rgb=(0.2651445 , 0.23295593, 0.5165993 )
			rgb=(0.2636632 , 0.23763078, 0.51876163)
			rgb=(0.26213801, 0.24228619, 0.52083736)
			rgb=(0.26057103, 0.2469217 , 0.52282822)
			rgb=(0.25896451, 0.25153685, 0.52473609)
			rgb=(0.25732244, 0.2561304 , 0.52656332)
			rgb=(0.25564519, 0.26070284, 0.52831152)
			rgb=(0.25393498, 0.26525384, 0.52998273)
			rgb=(0.25219404, 0.26978306, 0.53157905)
			rgb=(0.25042462, 0.27429024, 0.53310261)
			rgb=(0.24862899, 0.27877509, 0.53455561)
			rgb=(0.2468114 , 0.28323662, 0.53594093)
			rgb=(0.24497208, 0.28767547, 0.53726018)
			rgb=(0.24311324, 0.29209154, 0.53851561)
			rgb=(0.24123708, 0.29648471, 0.53970946)
			rgb=(0.23934575, 0.30085494, 0.54084398)
			rgb=(0.23744138, 0.30520222, 0.5419214 )
			rgb=(0.23552606, 0.30952657, 0.54294396)
			rgb=(0.23360277, 0.31382773, 0.54391424)
			rgb=(0.2316735 , 0.3181058 , 0.54483444)
			rgb=(0.22973926, 0.32236127, 0.54570633)
			rgb=(0.22780192, 0.32659432, 0.546532 )
			rgb=(0.2258633 , 0.33080515, 0.54731353)
			rgb=(0.22392515, 0.334994 , 0.54805291)
			rgb=(0.22198915, 0.33916114, 0.54875211)
			rgb=(0.22005691, 0.34330688, 0.54941304)
			rgb=(0.21812995, 0.34743154, 0.55003755)
			rgb=(0.21620971, 0.35153548, 0.55062743)
			rgb=(0.21429757, 0.35561907, 0.5511844 )
			rgb=(0.21239477, 0.35968273, 0.55171011)
			rgb=(0.2105031 , 0.36372671, 0.55220646)
			rgb=(0.20862342, 0.36775151, 0.55267486)
			rgb=(0.20675628, 0.37175775, 0.55311653)
			rgb=(0.20490257, 0.37574589, 0.55353282)
			rgb=(0.20306309, 0.37971644, 0.55392505)
			rgb=(0.20123854, 0.38366989, 0.55429441)
			rgb=(0.1994295 , 0.38760678, 0.55464205)
			rgb=(0.1976365 , 0.39152762, 0.55496905)
			rgb=(0.19585993, 0.39543297, 0.55527637)
			rgb=(0.19410009, 0.39932336, 0.55556494)
			rgb=(0.19235719, 0.40319934, 0.55583559)
			rgb=(0.19063135, 0.40706148, 0.55608907)
			rgb=(0.18892259, 0.41091033, 0.55632606)
			rgb=(0.18723083, 0.41474645, 0.55654717)
			rgb=(0.18555593, 0.4185704 , 0.55675292)
			rgb=(0.18389763, 0.42238275, 0.55694377)
			rgb=(0.18225561, 0.42618405, 0.5571201 )
			rgb=(0.18062949, 0.42997486, 0.55728221)
			rgb=(0.17901879, 0.43375572, 0.55743035)
			rgb=(0.17742298, 0.4375272 , 0.55756466)
			rgb=(0.17584148, 0.44128981, 0.55768526)
			rgb=(0.17427363, 0.4450441 , 0.55779216)
			rgb=(0.17271876, 0.4487906 , 0.55788532)
			rgb=(0.17117615, 0.4525298 , 0.55796464)
			rgb=(0.16964573, 0.45626209, 0.55803034)
			rgb=(0.16812641, 0.45998802, 0.55808199)
			rgb=(0.1666171 , 0.46370813, 0.55811913)
			rgb=(0.16511703, 0.4674229 , 0.55814141)
			rgb=(0.16362543, 0.47113278, 0.55814842)
			rgb=(0.16214155, 0.47483821, 0.55813967)
			rgb=(0.16066467, 0.47853961, 0.55811466)
			rgb=(0.15919413, 0.4822374 , 0.5580728 )
			rgb=(0.15772933, 0.48593197, 0.55801347)
			rgb=(0.15626973, 0.4896237 , 0.557936 )
			rgb=(0.15481488, 0.49331293, 0.55783967)
			rgb=(0.15336445, 0.49700003, 0.55772371)
			rgb=(0.1519182 , 0.50068529, 0.55758733)
			rgb=(0.15047605, 0.50436904, 0.55742968)
			rgb=(0.14903918, 0.50805136, 0.5572505 )
			rgb=(0.14760731, 0.51173263, 0.55704861)
			rgb=(0.14618026, 0.51541316, 0.55682271)
			rgb=(0.14475863, 0.51909319, 0.55657181)
			rgb=(0.14334327, 0.52277292, 0.55629491)
			rgb=(0.14193527, 0.52645254, 0.55599097)
			rgb=(0.14053599, 0.53013219, 0.55565893)
			rgb=(0.13914708, 0.53381201, 0.55529773)
			rgb=(0.13777048, 0.53749213, 0.55490625)
			rgb=(0.1364085 , 0.54117264, 0.55448339)
			rgb=(0.13506561, 0.54485335, 0.55402906)
			rgb=(0.13374299, 0.54853458, 0.55354108)
			rgb=(0.13244401, 0.55221637, 0.55301828)
			rgb=(0.13117249, 0.55589872, 0.55245948)
			rgb=(0.1299327 , 0.55958162, 0.55186354)
			rgb=(0.12872938, 0.56326503, 0.55122927)
			rgb=(0.12756771, 0.56694891, 0.55055551)
			rgb=(0.12645338, 0.57063316, 0.5498411 )
			rgb=(0.12539383, 0.57431754, 0.54908564)
			rgb=(0.12439474, 0.57800205, 0.5482874 )
			rgb=(0.12346281, 0.58168661, 0.54744498)
			rgb=(0.12260562, 0.58537105, 0.54655722)
			rgb=(0.12183122, 0.58905521, 0.54562298)
			rgb=(0.12114807, 0.59273889, 0.54464114)
			rgb=(0.12056501, 0.59642187, 0.54361058)
			rgb=(0.12009154, 0.60010387, 0.54253043)
			rgb=(0.11973756, 0.60378459, 0.54139999)
			rgb=(0.11951163, 0.60746388, 0.54021751)
			rgb=(0.11942341, 0.61114146, 0.53898192)
			rgb=(0.11948255, 0.61481702, 0.53769219)
			rgb=(0.11969858, 0.61849025, 0.53634733)
			rgb=(0.12008079, 0.62216081, 0.53494633)
			rgb=(0.12063824, 0.62582833, 0.53348834)
			rgb=(0.12137972, 0.62949242, 0.53197275)
			rgb=(0.12231244, 0.63315277, 0.53039808)
			rgb=(0.12344358, 0.63680899, 0.52876343)
			rgb=(0.12477953, 0.64046069, 0.52706792)
			rgb=(0.12632581, 0.64410744, 0.52531069)
			rgb=(0.12808703, 0.64774881, 0.52349092)
			rgb=(0.13006688, 0.65138436, 0.52160791)
			rgb=(0.13226797, 0.65501363, 0.51966086)
			rgb=(0.13469183, 0.65863619, 0.5176488 )
			rgb=(0.13733921, 0.66225157, 0.51557101)
			rgb=(0.14020991, 0.66585927, 0.5134268 )
			rgb=(0.14330291, 0.66945881, 0.51121549)
			rgb=(0.1466164 , 0.67304968, 0.50893644)
			rgb=(0.15014782, 0.67663139, 0.5065889 )
			rgb=(0.15389405, 0.68020343, 0.50417217)
			rgb=(0.15785146, 0.68376525, 0.50168574)
			rgb=(0.16201598, 0.68731632, 0.49912906)
			rgb=(0.1663832 , 0.69085611, 0.49650163)
			rgb=(0.1709484 , 0.69438405, 0.49380294)
			rgb=(0.17570671, 0.6978996 , 0.49103252)
			rgb=(0.18065314, 0.70140222, 0.48818938)
			rgb=(0.18578266, 0.70489133, 0.48527326)
			rgb=(0.19109018, 0.70836635, 0.48228395)
			rgb=(0.19657063, 0.71182668, 0.47922108)
			rgb=(0.20221902, 0.71527175, 0.47608431)
			rgb=(0.20803045, 0.71870095, 0.4728733 )
			rgb=(0.21400015, 0.72211371, 0.46958774)
			rgb=(0.22012381, 0.72550945, 0.46622638)
			rgb=(0.2263969 , 0.72888753, 0.46278934)
			rgb=(0.23281498, 0.73224735, 0.45927675)
			rgb=(0.2393739 , 0.73558828, 0.45568838)
			rgb=(0.24606968, 0.73890972, 0.45202405)
			rgb=(0.25289851, 0.74221104, 0.44828355)
			rgb=(0.25985676, 0.74549162, 0.44446673)
			rgb=(0.26694127, 0.74875084, 0.44057284)
			rgb=(0.27414922, 0.75198807, 0.4366009 )
			rgb=(0.28147681, 0.75520266, 0.43255207)
			rgb=(0.28892102, 0.75839399, 0.42842626)
			rgb=(0.29647899, 0.76156142, 0.42422341)
			rgb=(0.30414796, 0.76470433, 0.41994346)
			rgb=(0.31192534, 0.76782207, 0.41558638)
			rgb=(0.3198086 , 0.77091403, 0.41115215)
			rgb=(0.3277958 , 0.77397953, 0.40664011)
			rgb=(0.33588539, 0.7770179 , 0.40204917)
			rgb=(0.34407411, 0.78002855, 0.39738103)
			rgb=(0.35235985, 0.78301086, 0.39263579)
			rgb=(0.36074053, 0.78596419, 0.38781353)
			rgb=(0.3692142 , 0.78888793, 0.38291438)
			rgb=(0.37777892, 0.79178146, 0.3779385 )
			rgb=(0.38643282, 0.79464415, 0.37288606)
			rgb=(0.39517408, 0.79747541, 0.36775726)
			rgb=(0.40400101, 0.80027461, 0.36255223)
			rgb=(0.4129135 , 0.80304099, 0.35726893)
			rgb=(0.42190813, 0.80577412, 0.35191009)
			rgb=(0.43098317, 0.80847343, 0.34647607)
			rgb=(0.44013691, 0.81113836, 0.3409673 )
			rgb=(0.44936763, 0.81376835, 0.33538426)
			rgb=(0.45867362, 0.81636288, 0.32972749)
			rgb=(0.46805314, 0.81892143, 0.32399761)
			rgb=(0.47750446, 0.82144351, 0.31819529)
			rgb=(0.4870258 , 0.82392862, 0.31232133)
			rgb=(0.49661536, 0.82637633, 0.30637661)
			rgb=(0.5062713 , 0.82878621, 0.30036211)
			rgb=(0.51599182, 0.83115784, 0.29427888)
			rgb=(0.52577622, 0.83349064, 0.2881265 )
			rgb=(0.5356211 , 0.83578452, 0.28190832)
			rgb=(0.5455244 , 0.83803918, 0.27562602)
			rgb=(0.55548397, 0.84025437, 0.26928147)
			rgb=(0.5654976 , 0.8424299 , 0.26287683)
			rgb=(0.57556297, 0.84456561, 0.25641457)
			rgb=(0.58567772, 0.84666139, 0.24989748)
			rgb=(0.59583934, 0.84871722, 0.24332878)
			rgb=(0.60604528, 0.8507331 , 0.23671214)
			rgb=(0.61629283, 0.85270912, 0.23005179)
			rgb=(0.62657923, 0.85464543, 0.22335258)
			rgb=(0.63690157, 0.85654226, 0.21662012)
			rgb=(0.64725685, 0.85839991, 0.20986086)
			rgb=(0.65764197, 0.86021878, 0.20308229)
			rgb=(0.66805369, 0.86199932, 0.19629307)
			rgb=(0.67848868, 0.86374211, 0.18950326)
			rgb=(0.68894351, 0.86544779, 0.18272455)
			rgb=(0.69941463, 0.86711711, 0.17597055)
			rgb=(0.70989842, 0.86875092, 0.16925712)
			rgb=(0.72039115, 0.87035015, 0.16260273)
			rgb=(0.73088902, 0.87191584, 0.15602894)
			rgb=(0.74138803, 0.87344918, 0.14956101)
			rgb=(0.75188414, 0.87495143, 0.14322828)
			rgb=(0.76237342, 0.87642392, 0.13706449)
			rgb=(0.77285183, 0.87786808, 0.13110864)
			rgb=(0.78331535, 0.87928545, 0.12540538)
			rgb=(0.79375994, 0.88067763, 0.12000532)
			rgb=(0.80418159, 0.88204632, 0.11496505)
			rgb=(0.81457634, 0.88339329, 0.11034678)
			rgb=(0.82494028, 0.88472036, 0.10621724)
			rgb=(0.83526959, 0.88602943, 0.1026459 )
			rgb=(0.84556056, 0.88732243, 0.09970219)
			rgb=(0.8558096 , 0.88860134, 0.09745186)
			rgb=(0.86601325, 0.88986815, 0.09595277)
			rgb=(0.87616824, 0.89112487, 0.09525046)
			rgb=(0.88627146, 0.89237353, 0.09537439)
			rgb=(0.89632002, 0.89361614, 0.09633538)
			rgb=(0.90631121, 0.89485467, 0.09812496)
			rgb=(0.91624212, 0.89609127, 0.1007168 )
			rgb=(0.92610579, 0.89732977, 0.10407067)
			rgb=(0.93590444, 0.8985704 , 0.10813094)
			rgb=(0.94563626, 0.899815 , 0.11283773)
			rgb=(0.95529972, 0.90106534, 0.11812832)
			rgb=(0.96489353, 0.90232311, 0.12394051)
			rgb=(0.97441665, 0.90358991, 0.13021494)
			rgb=(0.98386829, 0.90486726, 0.13689671)
			rgb=(0.99324789, 0.90615657, 0.1439362 )
		}
}

\pgfplotsset{
	colormap={viridisSoft}{
			rgb255=(242, 242, 242);
			rgb=(0.28026,0.1657,0.4765);
			rgb=(0.26366,0.23763,0.51877);
			rgb=(0.23744,0.3052,0.54192);
			rgb=(0.20862,0.36775,0.55267);
			rgb=(0.18225,0.42618,0.55711);
			rgb=(0.1592,0.48224,0.55807);
			rgb=(0.13777,0.53749,0.5549);
			rgb=(0.12115,0.59274,0.54465);
			rgb=(0.12808,0.64775,0.5235);
			rgb=(0.18065,0.7014,0.48819);
			rgb=(0.27415,0.75198,0.4366);
			rgb=(0.39517,0.79747,0.36775);
			rgb=(0.53561,0.83578,0.2819);
			rgb=(0.68895,0.86545,0.18272);
			rgb=(0.84557,0.88733,0.0997);
			rgb=(0.99324,0.90616,0.14394)
		}
}

\pgfplotsset{
	colormap={cellRed}{
			rgb255=(242.0,242.0,242.0);
			rgb255=(241.63157894736844,234.47368421052633,234.47368421052633);
			rgb255=(241.26315789473685,226.94736842105266,226.94736842105266);
			rgb255=(240.89473684210526,219.42105263157893,219.42105263157893);
			rgb255=(240.5263157894737,211.89473684210526,211.89473684210526);
			rgb255=(240.1578947368421,204.3684210526316,204.3684210526316);
			rgb255=(239.78947368421052,196.84210526315792,196.84210526315792);
			rgb255=(239.42105263157896,189.31578947368422,189.31578947368422);
			rgb255=(239.05263157894737,181.78947368421052,181.78947368421052);
			rgb255=(238.6842105263158,174.26315789473688,174.26315789473688);
			rgb255=(238.31578947368422,166.73684210526315,166.73684210526315);
			rgb255=(237.94736842105263,159.21052631578948,159.21052631578948);
			rgb255=(237.57894736842104,151.68421052631578,151.68421052631578);
			rgb255=(237.21052631578948,144.1578947368421,144.1578947368421);
			rgb255=(236.84210526315792,136.63157894736844,136.63157894736844);
			rgb255=(236.47368421052633,129.10526315789474,129.10526315789474);
			rgb255=(236.10526315789474,121.57894736842107,121.57894736842107);
			rgb255=(235.73684210526318,114.05263157894737,114.05263157894737);
			rgb255=(235.3684210526316,106.52631578947368,106.52631578947368);
			rgb255=(235.0,99.0,99.0);
		}
}

\pgfplotsset{
	colormap={cellGreen}{
			rgb255=(242.0,242.0,242.0);
			rgb255=(236.21052631578948,239.5263157894737,234.26315789473685);
			rgb255=(230.42105263157896,237.05263157894737,226.5263157894737);
			rgb255=(224.6315789473684,234.57894736842104,218.78947368421052);
			rgb255=(218.8421052631579,232.10526315789474,211.05263157894737);
			rgb255=(213.05263157894737,229.63157894736844,203.31578947368422);
			rgb255=(207.26315789473685,227.1578947368421,195.57894736842107);
			rgb255=(201.4736842105263,224.68421052631578,187.8421052631579);
			rgb255=(195.68421052631578,222.21052631578948,180.10526315789474);
			rgb255=(189.8947368421053,219.73684210526318,172.36842105263162);
			rgb255=(184.10526315789474,217.26315789473682,164.63157894736844);
			rgb255=(178.31578947368422,214.78947368421052,156.89473684210526);
			rgb255=(172.5263157894737,212.31578947368422,149.1578947368421);
			rgb255=(166.73684210526318,209.84210526315792,141.42105263157896);
			rgb255=(160.94736842105263,207.3684210526316,133.6842105263158);
			rgb255=(155.1578947368421,204.89473684210526,125.94736842105263);
			rgb255=(149.3684210526316,202.42105263157893,118.21052631578948);
			rgb255=(143.57894736842104,199.94736842105266,110.47368421052632);
			rgb255=(137.78947368421052,197.47368421052633,102.73684210526316);
			rgb255=(132.0,195.0,95.0);
		}
}

\pgfplotsset{
	colormap={cellRedSquared}{
			rgb255=(242.0,242.0,242.0);
			rgb255=(241.28254847645428,227.34349030470915,227.34349030470915);
			rgb255=(240.60387811634348,213.47922437673128,213.47922437673128);
			rgb255=(239.9639889196676,200.40720221606648,200.40720221606648);
			rgb255=(239.36288088642658,188.1274238227147,188.1274238227147);
			rgb255=(238.8005540166205,176.63988919667594,176.63988919667594);
			rgb255=(238.2770083102493,165.94459833795014,165.94459833795014);
			rgb255=(237.79224376731304,156.04155124653738,156.04155124653738);
			rgb255=(237.34626038781164,146.93074792243766,146.93074792243766);
			rgb255=(236.93905817174516,138.61218836565098,138.61218836565098);
			rgb255=(236.57063711911357,131.0858725761773,131.0858725761773);
			rgb255=(236.2409972299169,124.35180055401662,124.35180055401662);
			rgb255=(235.95013850415512,118.40997229916897,118.40997229916897);
			rgb255=(235.69806094182823,113.26038781163435,113.26038781163435);
			rgb255=(235.4847645429363,108.90304709141274,108.90304709141274);
			rgb255=(235.3102493074792,105.33795013850416,105.33795013850416);
			rgb255=(235.17451523545705,102.56509695290858,102.56509695290858);
			rgb255=(235.0775623268698,100.58448753462605,100.58448753462605);
			rgb255=(235.01939058171746,99.3961218836565,99.3961218836565);
			rgb255=(235.0,99.0,99.0);
		}
}
\pgfplotsset{
	colormap={cellGreenSquared}{
			rgb255=(242.0,242.0,242.0);
			rgb255=(230.7257617728532,237.18282548476455,226.93351800554018);
			rgb255=(220.06094182825484,232.62603878116343,212.6814404432133);
			rgb255=(210.00554016620498,228.32963988919667,199.2437673130194);
			rgb255=(200.5595567867036,224.29362880886427,186.62049861495845);
			rgb255=(191.7229916897507,220.5180055401662,174.8116343490305);
			rgb255=(183.49584487534625,217.0027700831025,163.81717451523545);
			rgb255=(175.87811634349032,213.74792243767314,153.63711911357342);
			rgb255=(168.86980609418282,210.75346260387812,144.27146814404432);
			rgb255=(162.47091412742384,208.01939058171746,135.72022160664818);
			rgb255=(156.68144044321332,205.54570637119116,127.98337950138506);
			rgb255=(151.50138504155126,203.33240997229916,121.06094182825484);
			rgb255=(146.9307479224377,201.37950138504155,114.95290858725764);
			rgb255=(142.96952908587255,199.68698060941827,109.65927977839334);
			rgb255=(139.61772853185596,198.25484764542935,105.18005540166205);
			rgb255=(136.8753462603878,197.0831024930748,101.51523545706371);
			rgb255=(134.74238227146813,196.17174515235456,98.66481994459834);
			rgb255=(133.21883656509695,195.5207756232687,96.62880886426592);
			rgb255=(132.30470914127426,195.13019390581718,95.40720221606648);
			rgb255=(132.0,195.0,95.0);
		}
} 
\usepgfplotslibrary{patchplots}

\pgfplotsset{every axis/.append style={
			grid=both,
			grid style={white, line width=.1pt},
			major grid style={white, line width=1.5pt},
			axis background/.style={fill=gray!10},
			axis line style={draw=none},
			tick style={draw=none},
			xlabel = $x$,
			line width=1pt,
			legend style={
					line width = 1pt,
					draw=none,
					/tikz/every even column/.append style={column sep=0.5cm}
				},
		}}

\usepgfplotslibrary{groupplots}

\usepackage{calc}

\definecolor{gg0}{HTML}{E24A33}
\definecolor{gg1}{HTML}{348ABD}
\definecolor{gg2}{HTML}{988ED5}
\definecolor{gg3}{HTML}{777777}
\definecolor{gg4}{HTML}{FBC15E}
\definecolor{gg5}{HTML}{8EBA42}
\definecolor{gg6}{HTML}{FFB5B8}

\pgfplotsset{
	/pgfplots/colormap={bright}{rgb255=(0,0,0) rgb255=(78,3,100) rgb255=(2,74,255)
			rgb255=(255,21,181) rgb255=(255,113,26) rgb255=(147,213,114) rgb255=(230,255,0)
			rgb255=(255,255,255)}
}

\newcommand{\pxx}{\partial_{xx}}

\newcommand{\addappendix}{
	\section*{\appendixname}
	\addcontentsline{toc}{section}{\appendixname}
	\counterwithin*{figure}{section}
	\stepcounter{section}
	\renewcommand{\thesection}{A}
	\renewcommand{\thefigure}{\thesection.\arabic{figure}}
}

\def\bm{\boldsymbol}

\newcommand{\ad}{^{\textrm{ad}}}
\newcommand{\dif}{^{\textrm{dif}}}

\definecolor{brandeisblue}{rgb}{0.0, 0.44, 1.0}
\definecolor{lincolngreen}{rgb}{0.11, 0.35, 0.02}
\definecolor{indiagreen}{rgb}{0.07, 0.53, 0.03}
\definecolor{venetianred}{rgb}{0.78, 0.03, 0.08}
\definecolor{darkorange}{rgb}{1.0, 0.55, 0.0}
\definecolor{burntorange}{rgb}{0.8, 0.33, 0.0}
\definecolor{flame}{rgb}{0.89, 0.35, 0.13}
\definecolor{non-photoblue}{rgb}{0.64, 0.87, 0.93}  

\renewcommand{\review}[2]{}
\renewcommand{\creview}[3]{}
\renewcommand{\ntcreview}[3]{}

\renewcommand{\tableofcontents}{}
\renewcommand{\listofreviews}{}

\expandafter\def\csname ver@etex.sty\endcsname{3000/12/31}

\usepackage{autonum}
\makeatletter
\patchcmd{\autonum@saveEnvironmentSubcommands}
{(0,0)\begin}
{(0,0)\hfuzz=\maxdimen\begin}
{}{}
\makeatother

\AtBeginDocument{

}

\makeatletter
\autonum@generatePatchedReferenceCSL{ref}
\autonum@generatePatchedReferenceCSL{eqref}
\autonum@generatePatchedReferenceCSL{Cref}
\makeatother

\definecolor{revisionColourOne}{RGB}{180,0,0}
\definecolor{revisionColourTwo}{RGB}{0,0,180}

\usepackage{setspace}
\newcommand{\st}{_\infty}

\begin{document}
\begin{singlespace}\maketitle\end{singlespace}
\begin{abstract}
	We propose a new fractional Laplacian for bounded domains, expressed as a conservation law and thus particularly suited to finite-volume schemes. Our approach permits the direct prescription of no-flux boundary conditions. We first show the well-posedness theory for the fractional heat equation. We also develop a numerical scheme, which correctly captures the action of the fractional Laplacian and its anomalous diffusion effect. We benchmark numerical solutions for the L\'{e}vy-Fokker-Planck equation against known analytical solutions. We conclude by numerically exploring properties of these equations with respect to their stationary states and long-time asymptotics.
\end{abstract}
 \subjectclassification{\subjectPDF}
\tableofcontents
\listofreviews

\section{Introduction}

The aim of this work is the design of a finite-volume numerical scheme to approximate the solution of the non-local diffusion problem given by the fractional heat equation and the related L\'{e}vy-Fokker-Planck equation. The fractional heat equation is defined in $\Rd$ as
\begin{equation}\label{eq:fracHE}
	\pder{\rho}{t} = - \prt{-\laplace}^{\frac{\alpha}{2}}\rho
\end{equation}
for $0 < \alpha \leq 2$. The so-called \textit{fractional Laplacian}, $\prt{-\laplace}^{\frac{\alpha}{2}}\rho$, can be formally defined by its Fourier symbol $\abs{\xi}^\alpha \hat\rho$, although it admits up to ten equivalent definitions (see \cite{Kwasnicki2017}). There is a suitable self-similar change of variables that leads to $\pder{\rho}{t} = \nabla \cdot (x\rho) - \prt{-\laplace}^{\frac{\alpha}{2}}\rho$, a particular case of L\'{e}vy-Fokker-Planck equation given by
\begin{equation}\label{eq:fracFP}
	\pder{\rho}{t} = \nabla \cdot (\beta x\rho) - \prt{-\laplace}^{\frac{\alpha}{2}}\rho
\end{equation}
for $0 < \alpha \leq 2$ and $\beta\geq 0$. Notice that this equation generalises the usual Fokker-Planck equation $\pder{\rho}{t} = \nabla \cdot (\beta x\rho) + \laplace \rho$ by replacing the Laplacian with a fractional operator, see \cite{biler2003generalised,gentil2008levy}.

Fractional diffusion (in particular, the fractional Laplacian) has been shown to be the mean field limit of L\'{e}vy walks under certain scalings \cite{Valdinoci2009}. This kind of stochastic process consists in the random movement of particles in space, subject to a probability that allows long jumps with a polynomial tail. Such random walks are long range stochastic processes and they are generally considered more realistic in the modelling of certain biological phenomena \cite{E06,LR09,BC10,DiNezzaPalatucciValdinoci2012,LS19}.

The inverse Fourier transform of the symbol $\abs{\xi}^\alpha \hat\rho$ yields, after some work, the \textit{Riesz} or \textit{singular integral} definition of the fractional Laplacian:
\begin{align}\label{eq:singularFL}
	\prt{-\laplace}^{\frac{\alpha}{2}} \rho(x) \coloneqq
	\curlyC\prt{d,\alpha}
	\;\textrm{p.v.}
	\int_\Rd
	\frac{\rho(x)-\rho(y)}{\abs*{x-y}^{d+\alpha}}
	\dy,
\end{align}
where the integral is understood in the Cauchy principal value sense in order to overcome the singularity. The constant $\curlyC\prt{d,\alpha}$, a term which arises in the computation of the inverse transform of $\abs{\xi}^\alpha$, is given by
\begin{align}\label{eq:constantFL}
	\curlyC\prt{d,\alpha}
	=
	\frac{
		2^\alpha \Gamma\prt*{\frac{d+\alpha}{2}}
	}{
		\pi^{\frac{d}{2}} \abs*{\Gamma\prt*{-\frac{\alpha}{2}}}
	}.
\end{align}
Using the Riesz potential, \cref{eq:fracFP} can be formally written in divergence form as
\begin{equation}\label{eq:fluxFP}
	\pder{\rho}{t} + \div F = 0,
	\quad\text{where }
	F =
	-\beta x\rho
	+
	\grad \brk*{
	\prt{-\laplace}^{\frac{\alpha}{2}-1} \rho
	}.
\end{equation}
The advantage of this form is that the fractional operator now appears with a negative exponent. In this case, the inverse Fourier transform of the symbol $\abs{\xi}^{-\alpha} \hat\rho$ yields (see \cite[Chapter 5]{Stein1970})
\begin{align}\label{eq:inverseFL}
	\prt{-\laplace}^{-\frac{\alpha}{2}} \rho(x) =
	\curlyC\prt{d,-\alpha}
	\int_\Rd
	\frac{\rho(y)}{\abs*{x-y}^{d-\alpha}}
	\dy
\end{align}
whenever $0<\alpha<d$. This form for the inverse operator bypasses the singularity altogether. \Cref{eq:fluxFP} can therefore be rewritten as
\begin{equation}\label{eq:fluxFP2}
	\pder{\rho}{t} + \div F = 0,
	\quad\text{where }
	F =
	-\beta x\rho
	+
	\curlyC\prt{d,\alpha-2}
	\grad
	\int_\Rd
	\frac{\rho(y)}{\abs*{x-y}^{d+\alpha-2}}
	\dy
\end{equation}
whenever $\alpha > 2-d$. Therefore, in dimension one this formulation is only valid for $1 < \alpha \leq 2$; in higher dimensions, for $0 < \alpha \leq 2$. This new form of the equation has two advantages: the first, that the fractional operator is no longer singular; and the second, that an equation in divergence form lends itself to be discretised in the finite-volume fashion. Finite volume schemes have been used with success to produce structure preserving schemes for equations in divergence form of gradient flow type and related systems, see \cite{CCH2015,BCH2020,BCM2020,CFS20}. This is a departure from the numerical methods for fractional diffusions that have been developed in the past, where the literature has been focused on finite-element and finite-difference methods
\cite{dT14,HuangOberman2014,NOS15,AB2017,CdTGP18,AG2018,BBN2018,LPG2020}. We also highlight several spectral methods \cite{mao2017hermite,sheng2020fast,cayama2021pseudospectral,xu2021asymptotic} which deal exclusively with problems on unbounded domains.

For the sake of computation, we would like to pose \cref{eq:fracFP} on an open bounded domain $\Omega\subset\Rd$. There are several non-equivalent definitions of fractional-type Laplacians on bounded domains that can be obtained as suitable restrictions of the definitions in $\Rd$ (see \cite[Section 1.2]{AbatangeloGomezCastroVazquez2022} and the references therein). In this work, we construct a new fractional Laplacian by restricting \cref{eq:fluxFP2} to the domain $\Omega$, prescribing zero-flux conditions for the divergence, and extending the density as $\rho\equiv 0$ on $\Rd\setminus\Omega$ in order to ensure that the non-local operator is well-defined. Thus, our interpretation of the L\'{e}vy-Fokker-Planck equation on a bounded domain is
\begin{equation}\label{eq:BcFP}
	\left \{
	\begin{aligned}
		 & \pder{\rho}{t} + \div F = 0,     \\
		 & F = -\beta x\rho
		+
		\curlyC\prt{d,\alpha}
		\grad
		\int_\Omega
		\frac{\rho(y)}{\abs*{x-y}^{d+\alpha-2}}
		\dy,                                \\
		 & F \cdot n_{\partial \Omega} = 0.
	\end{aligned}
	\right.
\end{equation}
In the absence of a drift term (when $\beta=0$), we also obtain an interpretation of the fractional heat equation on a bounded domain.
Notice, however, that the steady states of this problem will satisfy
\begin{equation}
	\int_\Omega
	\frac{\rho_\infty(y)}{\abs*{x-y}^{d+\alpha-2}}
	\dy = C,
\end{equation}
which causes $\rho_\infty$ to be singular on the boundary. However, for $\beta > 0$, our numerical results on suitably scaled quadrangular domains show that the numerical steady state is very similar to the self-similar profile of the $\Rd$ case (see \cref{sec:steady2D,sec:steady1D}).

The rest of this work is organised as follows: in \cref{sec:theory} we study the well-posedness of \cref{eq:BcFP}, distinguishing the cases $\beta=0$ and $\beta>0$; in \cref{sec:numericalSchemes} we introduce a finite-volume numerical scheme for \cref{eq:BcFP} in one dimension, and then generalise it to higher dimensions via dimensional splitting; we conclude in \cref{sec:numericalExperiments} by validating our schemes against known analytical results.
 \section{A new fractional Laplacian in bounded domains}\label{sec:theory}

The aim of this section is to establish a well-posedness theory for the new fractional operator on a bounded domain introduced above. We first define the Riesz kernel and the extension operator
\begin{equation}
	\mathcal I_{\gamma} [u] (x) = \mathcal{C}(d, -\gamma) \int_{\mathbb{R}^{d}} \frac{u(y)}{|x-y|^{d-\gamma}}d y,
	\qquad \qquad
	\mathcal E [u] (x) =
	\begin{dcases}
		u(x) & \text{if } x \in \Omega, \\
		0    & \text{otherwise},
	\end{dcases}
\end{equation}
where we assume that $0 < \gamma < d$ for the operator to be well defined. The operator $\mathcal I_{\gamma} : L^p (\mathbb R^d) \to L^q(\mathbb R^d)$ has been widely studied. Note that the flux in \cref{eq:BcFP} reduces, whenever $\beta = 0$, to
\begin{equation}
	F = \nabla \Big( \mathcal I_{2(1-\frac{\alpha}{2})} \mathcal E u \Big).
\end{equation}
Thus, besides $\alpha \in (0,2)$, we require $0 < 2(1-\frac{\alpha}{2}) < d$ for well-posedness, i.e. $\alpha > 2-d$. This is only restrictive in dimension $d = 1$.

Let $\mathcal B = \mathcal I_{2(1-\frac{\alpha}{2})} \mathcal E$. If $\Omega = \Rd$, then $\mathcal B = (-\Delta)^{-(1-\frac \alpha 2)}$, an inverse fractional Laplacian. Our new diffusion operator is therefore $\mathcal A = -\Delta \mathcal B u$, and we denote its domain by $D(\mathcal{A})$. \Cref{eq:BcFP} can be now written in the $\beta = 0$ case as
\begin{equation}
	\label{eq:heat}
	\begin{dcases}
		\frac{\partial u}{\partial t} = \Delta \mathcal B u & \text{in } (0,\infty) \times \Omega, \\
		\frac{\partial (\mathcal B u) }{\partial n} = 0     & \text{on } (0,\infty) \times \Omega.
	\end{dcases}
\end{equation}

\begin{theorem}
	There exists a unique semigroup $\mathcal S(t) : L^2 (\Omega) \to L^2 (\Omega)$ of solutions of \eqref{eq:heat}. In fact, if $\mathcal B u_0 \in H^2(\Omega) $ then $\mathcal B u(t) \in H^2 (\Omega)$ for all times and the equation is satisfied in the operator sense.
\end{theorem}

\begin{proof}
	We will first justify the well-posedness of this problem using the Hille-Yosida theorem applied to the operator $\mathcal A$ with the Neumann condition in $L^2 (\Omega)$ (see \cite[Theorem 7.4]{Brezis2010}).

	Let us construct $D(\mathcal A)$. We begin by remarking that $\mathcal B: L^2 (\Omega) \to L^2(\Omega)$ is self-adjoint, since
	\begin{equation}
		\int_\Omega v(x) \mathcal B u (x) \dx
		= \int_\Omega \int_\Omega \frac {u(y) v(x)} {|x-y|^{d-2(1-\frac{\alpha}{2})}} \dx \dy
		= \int_\Omega u(y) \mathcal B v (y) \dy.
	\end{equation}
	Since $\mathcal I_\alpha$ is a compact operator on $L^2_c (\Rd)$, so is $\mathcal B$ in $L^2 (\Omega)$. Thus, by the spectral theorem, there exists a basis of $L^2(\Omega)$ of orthonormal eigenfunctions $\varphi_i$ of $\mathcal B$ with eigenvalues $\lambda_i \to 0$, and,
	defining $u_i = \int_\Omega u \varphi_i \dx$, it holds
	\begin{equation}
		u(x) = \sum_{i=1}^\infty u_i \varphi_i(x) , \qquad \text{and} \qquad \mathcal B u(x) = \sum_{i=1}^\infty \lambda_i u_i \varphi_i(x) .
	\end{equation}
	Furthermore, with this construction $\|u\|_{L^2} = \sum_i |u_i|^2$.

	Let us show that $\lambda_i \ge 0$. Defining $U \defeq \mathcal B u$, notice that $(-\Delta)^{1-\frac{\alpha}{2}} U = \mathcal E (u)$ in $\Rd$. Therefore, for $u \in C_c^\infty (\Omega)$, we have
	\begin{equation}
		\int_\Omega u \mathcal B u = \int_\Omega U (-\Delta)^{1-\frac{\alpha}{2}} U = \int_\Rd U (-\Delta)^{1-\frac{\alpha}{2}} U = \int_\Rd \left|(-\Delta)^{ \frac{1-\frac{\alpha}{2}}{2} } U \right|^2 \ge 0.
	\end{equation}
	Hence $\lambda_i \ge 0$.

	We now show that $\lambda_i > 0$ for all $i$. Suppose, to the contrary, that $\lambda_i = 0$ for some $i$. We therefore have that $U \defeq \mathcal B \varphi_i = 0$. But then, a.e. in $\Omega$ we have that $\varphi_i = \mathcal E(\varphi_i) = (-\Delta)^{1-\frac{\alpha}{2}} U = 0$, a contradiction.

	Therefore, we can formally define the operator $\mathcal B^{-1}$ through the series $\mathcal B^{-1}u (x) = \sum_i \lambda_i^{-1} u_i \varphi_i(x)$.
	We define
	\begin{equation}
		D(\mathcal A) = \left\{ u = \mathcal B^{-1} v : v \in H^2 (\Omega), \nabla v \cdot n = 0 \text{ on } \partial \Omega \text{, and } \sum_i \lambda_i^{-2} |v_i|^2 < \infty \right\} .
	\end{equation}
	Notice, by construction, that $D(\mathcal A) = L^2 (\Omega) \cap \mathcal B^{-1} (H^2(\Omega))$. Since $\lambda_i > 0$, this set is not empty.
	Then $\mathcal A : D(\mathcal A) \subset L^2 (\Omega) \to L^2 (\Omega) $.

	Now we check that $\mathcal A$ is monotone. Take $u \in D(\mathcal A)$. Due the spectral decomposition of $\mathcal B$ it admits a square root and inverse square root $\mathcal B^{\pm \frac 1 2}$; take $w = \mathcal B^{\frac 1 2} u$. Then, due to the Neumann boundary condition
	\begin{equation}
		\int_\Omega u (\mathcal Au) = \int_\Omega \nabla u \cdot \nabla \mathcal Au = \int_\Omega |\nabla \mathcal B^{\frac 1 2} w|^2 \ge 0.
	\end{equation}

	Lastly, we check that $\mathcal A$ is maximal monotone. Take $f \in L^2(\Omega)$; we want to show there exists $u \in D(\mathcal A)$ such that $u + \mathcal A u = f$. Consider the weak formulation
	\begin{equation}
		\int_\Omega u \varphi + \int_\Omega \nabla \mathcal B u \cdot \nabla \varphi = \int_\Omega f \varphi , \qquad \forall \varphi \in H^1 (\Omega).
	\end{equation}
	Letting again $w = \mathcal B^{\frac 1 2} u$ and $\psi = \mathcal B^{-\frac 1 2} \varphi$, we obtain
	\begin{equation}
		\int_\Omega w \psi + \int_\Omega \nabla \mathcal B^{\frac 1 2} w \cdot \nabla \mathcal B^{\frac 1 2} \psi = \int_\Omega \mathcal B^{-\frac 1 2} w \mathcal B^{\frac 1 2} \psi + \int_\Omega \nabla \mathcal B^{\frac 1 2} w \cdot \nabla \mathcal B^{\frac 1 2} \psi = \int_\Omega f \mathcal B^{\frac 1 2} \psi .
	\end{equation}
	This is a problem of the form $a(w,\psi) = L(\psi)$ where $w,\psi \in V = \mathcal B^{-\frac 1 2}(H^1 (\Omega)) \cap L^2 (\Omega)$, where the bilinear form $a$ is symmetric and continuous in $V$. Hence, it can be solved using the Lax-Milgram theorem. A posteriori, it is trivial to verify that $u \in D(\mathcal A)$.

	We now satisfy all the hypotheses of the Hille-Yosida theorem. Thus, if $u_0 \in D(\mathcal A)$, then $\mathcal S(t) u_0$ is a solution in the strong sense, i.e. $u \in C([0,\infty), D(A)) \cap C^1([0,\infty), L^2(\Omega))$ and the equation is satisfied.
\end{proof}

The problem \eqref{eq:BcFP} is not purely diffusive when $\beta > 0$, hence the existence does not follow directly from the Hille-Yosida (or Lumer-Phillips) theorems. Unlike in $\Rd$, it cannot be deduced from the diffusive problem by a change of variables; that would lead to a domain $\Omega_t$ that evolves in time. Thus the theory of well-posedness for \eqref{eq:BcFP} when $\beta > 0$ is an open problem. A sensible approach would be to prove the convergence of our numerical scheme below.

 \section{Numerical schemes}\label{sec:numericalSchemes}

The thrust of this work is the discretisation of the fractional Laplacian term in \cref{eq:BcFP}. First we introduce the scheme in one spatial dimension, in order to highlight the technique used to approximate the fractional term, and then we generalise it to higher dimensions. Our discretisation of the advection term follows previous finite-volume works for generalised Fokker-Planck equations.

\subsection{One dimension}\label{sec:numericalSchemes1D}

In one dimension, we construct a scheme for \cref{eq:BcFP} in the range $1<\alpha<2$. We consider, without loss of generality, a domain $\Omega = (-R,R)$, and divide it into $N$ cells $C\i = \brk{x\imh, x\ih}$, for $i=1,\cdots,N$. Each cell is centred at the points $x\i$, where $x\i=-R+(i-1/2)\Dx$. For simplicity, we assume a uniform grid with cell size $\Dx = 2RN^{-1}$.

We denote by $\bar{\rho}\i(t)$ the average of the solution $\rho(t,x)$ over the $i$-th cell:
\[ \bar{\rho}\i(t) = \frac{1}{\Dx} \int_{C\i} \rho(t, x) \dx. \]
Then, equation \eqref{eq:BcFP} can be integrated on each cell $C\i$ to yield
\begin{equation}
	\der{\bar{\rho}\i(t)}{t} + \frac{F\brk*{\rho\prt{t,x\ih}} - F\brk*{\rho\prt{t,x\imh}}}{\Dx} = 0, \quad i=1, \cdots, N,
\end{equation}
which we approximate as
\begin{subequations}\label{eq:scheme1D}
\begin{align+}\label{ODEsys1D}
\der{\bar{\rho}\i(t)}{t}
+ \frac{F\ih(t) - F\imh(t)}{\Dx}
= 0,
\quad i=1, \cdots, N.
\end{align+}
The flux $F$ is split into an advective part $F\ad$ and a diffusive part $F\dif$:
\begin{align+}
F\ih(t) = F\ad\ih(t) + F\dif\ih(t).
\end{align+}
The advection flux corresponds to the discretisation of the Fokker-Planck term $\beta \div\prt{\rho x}$; here we follow the discretisation of \cite{CCH2015,BCH2020}:
\begin{align+}\label{eq:1DAdvectionFlux}
F\ad\ih(t)
&=
\bar\rho\i(t)\pos{v\ih}
+ \bar\rho\ip(t)\neg{v\ih},
\end{align+}
where
\begin{align+}\label{eq:1DAdvectionVelocity}
v\ih
=
-\frac{\xi\ip-\xi\i}{\Dx}
\text{ and }
\xi\i
=
\beta
\frac{\abs{x\i}^2}{2},
\end{align+}
for $\pos{s}=\max\set{0,s}$ and $\neg{s}=\min\set{0,s}$.

To treat the diffusion, the gradient term
$
	\curlyC\prt{1,\alpha-2}
	\grad
	\int_\Rd
	\rho(y)\abs*{x-y}^{1-\alpha}
	\dy
$
is replaced by the difference
\begin{align+}\label{eq:1DDiffusionFlux}
F\dif\ih(t)
=
\frac{I(t,x\ip) - I(t,x\i)}{\Dx}.
\end{align+}
The term $I(t,x\i)$ is the approximation of the integral $\mathcal I_{2-\alpha}[\rho](x\i)$, given as a discrete sum by
\begin{align+}\label{eq:integralTerms}
I(t,x\i)
& =
\sum_{k = 1}^{N}
\bar{\rho}\k(t)
I\k(x\i),
\quad\textrm{where}\quad
I\k(x\i)
\coloneqq
\mathcal{C}(1, \alpha - 2)
\int_{C\k} \abs*{x\i - y}^{1 - \alpha} \dy.
\end{align+}

To conclude, we impose no-flux boundary conditions:
\begin{align+}\label{eq:noflux1D}
F_{1-\nhalf}(t) = F_{N+\nhalf}(t) \equiv 0.
\end{align+}
\end{subequations}

\begin{remark}[Linearity]
	Scheme \eqref{eq:scheme1D} is linear. We can rewrite \cref{ODEsys1D} as a system
	\begin{equation}
		\der{\bbrho(t)}{t} + A \bbrho(t) = 0,
	\end{equation}
	for a vector
	$\bbrho =
		\begin{pmatrix}
			\bar\rho_1
			 & \cdots
			 & \bar\rho\i
			 & \cdots
			 & \bar\rho_N
		\end{pmatrix}^\top
	$.
	As with the flux, the constant matrix $A$ can be split into an advective part and a diffusive part, $A=A\ad+A\dif$. The advection matrix is simply
	\begin{align}\label{eq:advectionMatrix}
		A\ad & =
		\frac{\beta}{\Dx}
		\begin{pmatrix}
			\pos{v_{1+\nhalf}} & \neg{v_{1+\nhalf}}
			\\
			\ddots             & \ddots             & \ddots
			\\
			                   & -\pos{v\imh}       & -\neg{v\imh}+\pos{v\ih} & \neg{v\ih}
			\\
			                   &                    & \ddots                  & \ddots              & \ddots
			\\
			                   &                    &                         & -\pos{v_{N-\nhalf}} & -\neg{v_{N-\nhalf}}
		\end{pmatrix}
		.
	\end{align}
	The diffusion matrix can be written as the product of a discrete Laplacian and a dense matrix, $A\dif=LD$, where
	\begin{align}\label{eq:diffusionMatrix}
		L & =
		-\frac{1}{\Dx^2}
		\begin{pmatrix}
			-1 & 1
			\\
			1  & -2     & 1
			\\
			   & \ddots & \ddots & \ddots
			\\
			   &        & 1      & -2     & 1
			\\
			   &        &        & 1      & -1
		\end{pmatrix}
		, \quad
		D =
		\begin{pmatrix}
			I_1(x_1) & I_2(x_1) & \hdots & I_N(x_1)
			\\
			I_1(x_2) & I_2(x_2) & \hdots & I_N(x_2)
			\\
			\vdots   & \vdots   & \ddots & \vdots
			\\
			I_1(x_N) & I_2(x_N) & \hdots & I_N(x_N)
		\end{pmatrix}
		.
	\end{align}
\end{remark}

\begin{remark}[Symmetry]
	The terms in \eqref{eq:integralTerms} are symmetric, $I\k(x\i)=I\i(x\k)$. The matrix $D$ is thus symmetric, and it can be constructed by evaluating only $N$ terms.
\end{remark}

\begin{remark}[Higher order advection]\label{rm:higherOrder}
	The discretisation of the fractional diffusion term is consistent to second order, a fact which will be verified in \cref{sec:numericalExperiments}. However, the treatment of the advection term described above is only first-order accurate. A higher order discretisation (eg. flux limiters, MUSCL) may be used, at the expense of the linearity of the scheme. An example is presented in the Appendix.
\end{remark}

\begin{remark}[Time discretisation]
	In practice, we discretise scheme \eqref{eq:scheme1D} on the interval $t\in\prt{0,T}$ using a uniform step $\Delta t$. For the sake of stability, we employ the implicit time discretisation
	\begin{equation}
		\bbrho\mp = \prt{\textrm{Id}+\Dt A}^{-1}\bbrho\m,
	\end{equation}
	where $\textrm{Id}$ is the identity matrix. The update matrix $\prt{\textrm{Id}+\Dt A}^{-1}$ is computed once, offline, for each mesh size $(\Dt,\Dx)$, and then stored for successive use.
\end{remark}

\begin{remark}[The range $\alpha \leq 1$]
	The scheme presented in \cref{sec:numericalSchemes1D} is only valid in the range $1<\alpha<2$ due to the inversion formula \eqref{eq:inverseFL} used to rewrite the L\'{e}vy-Fokker-Planck equation as \eqref{eq:fluxFP2}. The range $0<\alpha\leq 1$ can be handled using instead the inversion formulae for the Poisson problem on the ball. The relevant kernels are given in \cite[Section 3]{Bucur2016}:
	\begin{align}
		I\k(x\i)
		 & =
		\frac{1}{\pi}
		\int_{x\kmh}^{x\kh}
		\log\prt*{
			\frac{
				R^2 - x\i y + \sqrt{\prt{R^2-x\i^2}\prt{R^2-y^2}}
			}{
				R\abs{x\i - y}
			}
		}\dy
		,\quad\text{for }\alpha=1;
		\\
		I\k(x\i)
		 & =
		\kappa\prt*{1, \alpha}
		\int_{x\kmh}^{x\kh}
		\abs{x\i - y}^{\alpha - 1}
		\int_0^{r_0\prt{x\i, y}}
		\frac{t^{\frac{\alpha}{2}-1}}{\prt{t+1}^{\frac{1}{2}}}
		\dt\dy
		,\quad\text{for } 0 < \alpha < 1;
	\end{align}
	where
	\begin{align}
		r_0\prt{x\i,y}
		=
		\frac{
			\prt{R^2-\abs{x\i}^2}\prt{R^2-\abs{y}^2}
		}{
			R^2\abs{x\i - y}^2
		},
		\quad
		\kappa\prt*{1, \alpha} = \frac{\Gamma(\frac{1}{2})}{2^{\alpha} \pi^{\frac{1}{2}} \Gamma^{2}(\frac{\alpha}{2})}.
	\end{align}
	Unfortunately, this approach renders the matrix $D$ no longer symmetric. We will not address this case directly.
\end{remark}

 \subsection{Two dimensions}

In two dimensions, the inversion formula \eqref{eq:inverseFL} no longer restricts the fractional exponent. Therefore, we construct a scheme for \cref{eq:BcFP} in the range $0<\alpha<2$.

We consider without loss of generality a square domain $\Omega = (-R,R)^2$, divided into $N^2$ cells given by $C\ij = \brk{x\imh, x\ih} \times \brk{y\jmh, y\jh}$, for $i,j=1,\cdots,N$. Each cell is centred at the points $\prt*{x\i,y\j}$, where $x\i=-R+(i-1/2)\Dx$, $y\j=-R+(j-1/2)\Dy$, and $\Dx = \Dy = 2RN^{-1}$. As in \cref{sec:numericalSchemes1D}, we approximate the cell averages of equation \eqref{eq:BcFP} to arrive at
\begin{subequations}\label{eq:scheme2D}
\begin{align+}\label{ODEsys2D}
\der{\bar{\rho}\ij(t)}{t}
+ \frac{F\ihj(t) - F\imhj(t)}{\Dx}
+ \frac{G\ijh(t) - G\ijmh(t)}{\Dy}
= 0,
\quad i, j = 1, \cdots, N.
\end{align+}
Once again, the fluxes are split into an advective part $F\ad$ and a diffusive part $F\dif$:
\begin{align+}
F\ihj(t) = F\ad\ihj(t) + F\dif\ihj(t), \quad
G\ijh(t) = G\ad\ijh(t) + G\dif\ijh(t).
\end{align+}
The advection terms are now
\begin{align+}\label{eq:2DAdvectionFlux}
F\ad\ihj(t)
& =
\bar\rho\ij(t)\pos{v\ihj}
+ \bar\rho\ipj(t)\neg{v\ihj},
\\
G\ad\ijh(t)
& =
\bar\rho\ij(t)\pos{w\ijh}
+ \bar\rho\ijp(t)\neg{w\ijh},
\end{align+}
where
\begin{align+}
v\ihj
=
-\frac{\xi\ipj-\xi\ij}{\Dx},\quad
w\ijh
=
-\frac{\xi\ijp-\xi\ij}{\Dy},\quad
\xi\ij
=
\beta \frac{\abs{x\i}^2+\abs{y\j}^2}{2},
\end{align+}
following \cite{CCH2015,BCH2020}.

The treatment of the diffusive part described in the previous section generalises to two dimensions:
\begin{align+}\label{DiffApprox2d}
F\ihj\dif(t)
=
\frac{I(t, x\ip, y\j) - I(t, x\i, y\j)}{\Dx}
\quad \textrm{and} \quad
G\ijh\dif(t)
=
\frac{I(t, x\i, y\jp) - I(t, x\i, y\j)}{\Dy}.
\end{align+}
The discrete integrals $I(t, x\i, y\j)$ are given by the sum
\begin{align+}\label{Int2DConv}
I(t, x\i, y\j)
=
\sum_{k, l = 1}^{N}
\bar{\rho}\kl(t) I\kl(x\i, y\j),
\end{align+}
where
\begin{align+}
I\kl(x\i, y\j)
\coloneqq
\mathcal{C}(2, \alpha - 2)
\int_{C\kl} \abs{\prt{x\i, y\j} - \prt{u,v}}^{-\alpha} \du\dv.
\end{align+}

Once again, we impose no-flux boundary conditions:
\begin{align+}
F_{1-\nhalf,\,j}(t) = F_{N+\nhalf,\,j}(t) \equiv 0,
\quad
G_{i,\,1-\nhalf}(t) = G_{i,\,N+\nhalf}(t) \equiv 0,
\quad
i, j = 1, \cdots, N.
\end{align+}
\end{subequations}

\subsubsection{Dimensional Splitting}\label{sec:splitting}

As in the one-dimensional case, scheme \eqref{eq:scheme2D} is linear. Writing
\begin{equation}\label{eq:linear_evolution}
	\der{\bbrho(t)}{t} + A \bbrho(t) = 0
\end{equation}
for a vector
$\bbrho =
	\begin{pmatrix}
		\bar{\rho}_{1,1}
		 & \bar{\rho}_{2,1}
		 & \cdots
		 & \bar{\rho}_{N,1}
		 & \bar{\rho}_{1,2}
		 & \cdots
		 & \bar{\rho}_{i,j}
		 & \cdots
		 & \bar{\rho}_{N,N}
	\end{pmatrix}^\top
$, we may split the matrix of the scheme into an advective part and a diffusive part, $A=A\ad+A\dif$, which are two-dimensional generalisations of \eqref{eq:advectionMatrix} and \eqref{eq:diffusionMatrix}.

However, the linear treatment of the scheme becomes impractical here, as the storage required for matrix of the scheme grows exponentially. While $A\ad$ is a banded matrix, $A\dif$ is dense. For $N=2^7$, a dense $N^2\times N^2$ matrix would require 2 gigabytes of RAM. For $N=2^8$, the size would be 16 gigabytes, or 128 gigabytes for $N=2^9$. In order to handle the computation, an approach which does not require the direct inversion of the matrix $\prt{\textrm{Id}+\Dt A}$ is required. The Krylov subspace methods, such as GMRES or BiCGSTAB \cite{Saad2003}, are among the available options. Instead, we resort to a \textit{dimensional splitting} strategy.

The (matrix) operator $A$ in \cref{eq:linear_evolution} is decomposed as $A = A_2 + A_1$, where $A_1$ corresponds to the transport terms along the $x$-direction (those in \cref{ODEsys2D} which arise from the fluxes $F\ihj$), and $A_2$ corresponds to the transport along the $y$-direction (terms related to $G\ijh$); naturally, $A_1$ and $A_2$ are independent of $\bbrho$, as is $A$. Formally, the solution to \eqref{eq:linear_evolution} can be written as $\bbrho(t) = \exp(tA) \bbrho_0$. One would like to approximate this by $\exp(tA_2)\exp(tA_1) \bbrho_0$ (i.e., by solving the problem one dimension at a time) but, in general, the solution operator cannot be factored in that way: $\exp(tA_2+tA_1) \neq \exp(tA_2)\exp(tA_1)$. However, the Lie-Trotter (or Trotter-Kato) formula
	\begin{align}
		\exp\prt*{tA_2+tA_1} =
		\lim\limits_{n\rightarrow\infty}
		\prt*{\exp\prt*{\frac{t}{n}A_2}\exp\prt*{\frac{t}{n}A_1}}^n
	\end{align}
	does hold for general square matrices \cite{trotter1959product} and some linear operators \cite{kato1978trotter}, and has been used to study the convergence of dimensional splitting in the case of linear semi-groups \cite{Ito_1998}. Choosing $\Dt = t n^{-1}$, we see that the exact solution operator $\exp\prt*{tA_2+tA_1}$ can be approximated by applying the operators $\exp\prt*{\Dt A_1}$ and $\exp\prt*{\Dt A_2}$ in an alternating sequence; i.e., the exact solution $\bbrho(t)$ can be approximated by performing a sequence of intermediate updates of $\bbrho_0$, each involving a short time, alternating the $x$-direction and $y$-direction sub-problems. Upon discretising time, the approximate solution $\bbrho\mp$ at time $(m+1)\Dt$ is computed from $\bbrho\m$ (that at time $m\Dt$) via $\bbrho\mh$, an intermediate step; $\bbrho\mh$ is computed from $\bbrho\m$ by solving the $x$-direction problem, and $\bbrho\mp$ is found from $\bbrho\mh$ by solving the $y$-problem.

At this stage, the advantage of the dimensional splitting approach is not clear: the matrices $A_1$ and $A_2$ are dense, as was $A$, so the memory requirement has effectively doubled. However, one further approximation is possible: each of the dimensional updates can be approximately decomposed row-wise or column-wise. For instance, to compute $\bbrho\mh$ from $\bbrho\m$: for each row $j$, compute $\bar{\rho}\mh\ij$ by solving the one-dimensional implicit problem within the row, assuming the value of the density will not change outside of it (i.e. $\bar{\rho}\mh\ik \equiv \bar{\rho}\m\ik$ whenever $k\neq j$). This is done independently on each row, and therefore can be trivially parallelised. A schematic diagram of the update is shown in \cref{fig:splitting}. Each update now involves the inversion of an $N\times N$ matrix, rather than $N^2\times N^2$, though the matrices are no longer independent of $\bbrho\m$. To obtain $\bbrho\mp$, the process is repeated along the $y$-direction, \textit{mutatis mutandis}. 

While this approach to dimensional splitting is partially justified by Lie-Trotter formula above, we will nevertheless justify it numerically in \cref{sec:numericalExperiments}, both in terms of checking the convergence of the scheme and its long-time behaviour.

\begin{figure}
	\centering
	\includegraphics{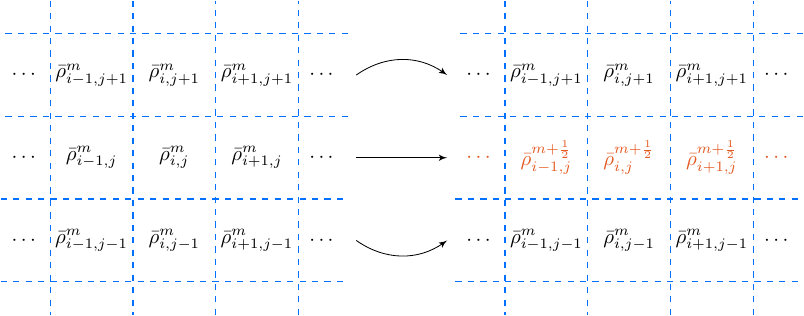}
	\caption{Dimensional splitting, row update. The split implicit problem considers information on the whole domain, but the density is allowed to change only within a single row. These updates take place independently for each row in parallel, and can be parallelised.}
	\label{fig:splitting}
\end{figure}

\begin{remark}[Sweeping dimensional splitting]
	A valid alternative is the \textit{sweeping dimensional splitting} described in \cite{BCH2020}. In that approach, the row and column updates take place sequentially, each considering the updated information from the previous step. This approach can be beneficial in some settings (it was used in \cite{BCH2020} to prove structural properties of the scheme), but was discarded here because it cannot be parallelised.
\end{remark}
 \section{Numerical experiments}\label{sec:numericalExperiments}

\FloatBarrier

We now demonstrate the accuracy and performance of our scheme in a variety of test cases, both in one and two dimensions. We will refer to the fractional heat equation \eqref{eq:fracHE} and the L\'{e}vy-Fokker-Plank equation \eqref{eq:fracFP} in the discussion; however, for the numerics, these are always understood as \cref{eq:BcFP} with $\beta=0$ and $\beta=1$, respectively.

In one dimension we employ scheme \eqref{eq:scheme1D}; in two dimensions we employ scheme \eqref{eq:scheme2D} with the dimensional splitting described in \cref{sec:splitting}. Experiments use the first-order upwind fluxes \eqref{eq:1DAdvectionFlux} and \eqref{eq:2DAdvectionFlux}, unless otherwise stated. The experiments that compute the order of accuracy of the scheme use instead the second-order \textit{minmod} flux \eqref{eq:minmodFlux} presented in the Appendix and discussed in \cref{rm:higherOrder}.

\subsection{One dimension}

\subsubsection{Fractional diffusion}
\label{sec:diffusion}

We first consider the fractional heat equation \eqref{eq:fracHE}. As in the classical heat equation, an explicit self-similar solution on the whole space is known when $\alpha=1$:
\begin{equation}
	\label{ExpKernFH}
	\phi(t,x) = C(d) \frac{t}{\prt*{t^{2}+|x|^{2}}^{\frac{d+1}{2}}}.
\end{equation}
Notice that the problem is linear. We pick $C(d)$ so that $\| \phi \|_{L^1} = 1$, i.e., $C(1) = \frac 1 \pi$ and $C(2) = \frac 1{2\pi}$.
We shall use this explicit solution to validate our numerical scheme. Technically, scheme \eqref{eq:scheme1D} is not valid for $\alpha = 1$; however, we can set $\alpha=1+\varepsilon$ and perform the comparison regardless. In practice, we choose $\varepsilon = 10^{-11}$.

\Cref{Fig1} shows a comparison of the numerical solution ($\alpha=1+\varepsilon$, $R=100$, $\Dx=0.1$, $\Dt=0.1$) on $\Omega=(-R,R)$ and the restriction of $\phi$ to $\Omega$. The initial datum is taken as $\phi(\Delta t, x)$. Both solutions match well on the interior of the domain; however there is a clear discrepancy on the boundary, where the numerical solution behaves singularly. The discrepancy is explained by the fact that the self-similar profile $\phi$ is \textit{leptokurtic} (i.e. has higher kurtosis, or thicker tails, than a Gaussian); therefore, the amount of mass that is ignored by considering $\phi$ on a bounded domain is never exponentially small. The singular behaviour at the boundary is a known effect of certain fractional operators \cite{abatangelo2023singular}. This effect is explored further in the next experiment.

\begin{figure}
	\centering

	\begin{subfigure}{0.49\textwidth}
		\centering
		\includegraphics[width=\textwidth]{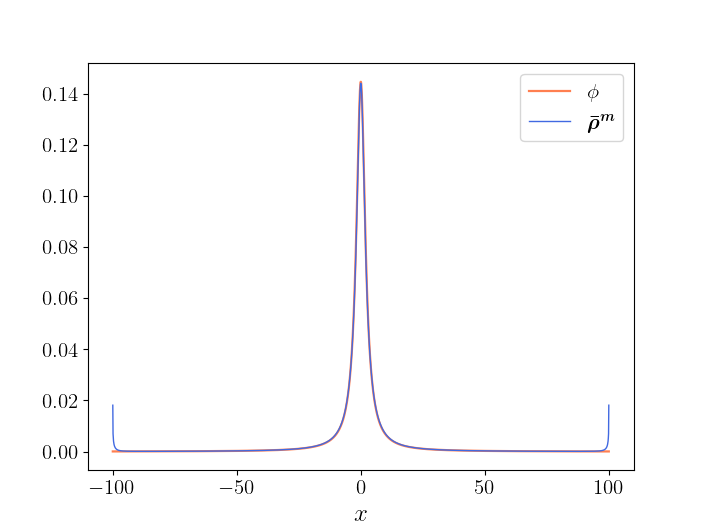}
		\caption{$t=2$}
	\end{subfigure}
	\hfill
	\begin{subfigure}{0.49\textwidth}
		\centering
		\includegraphics[width=\textwidth]{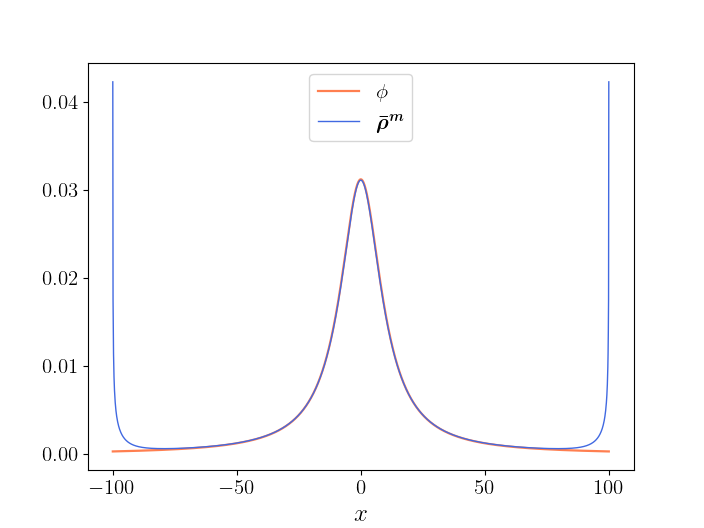}
		\caption{$t=10$}
	\end{subfigure}

	\begin{subfigure}{0.49\textwidth}
		\centering
		\includegraphics[width=\textwidth]{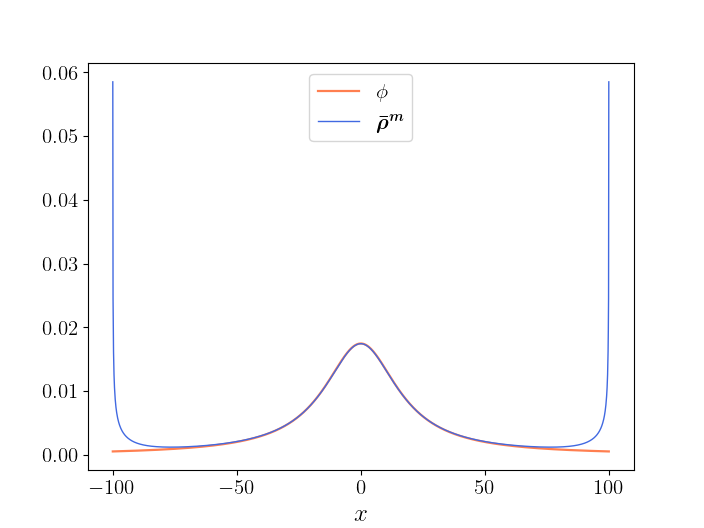}
		\caption{$t=18$}
	\end{subfigure}
	\hfill
	\begin{subfigure}{0.49\textwidth}
		\centering
		\includegraphics[width=\textwidth]{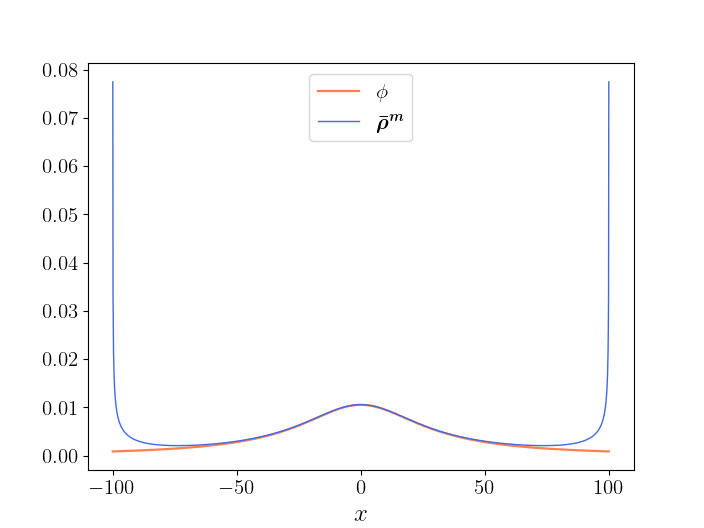}
		\caption{$t=30$}
	\end{subfigure}

	\caption{Fractional heat equation \eqref{eq:fracHE} in one dimension. Numerical solution $\bbrho\m$ on $\Omega=(-R,R)$ and explicit solution $\phi$ on $\R$. Scheme \eqref{eq:scheme1D}, $\alpha = 1 + \varepsilon$, $R=100$, $\Dx=0.1$, $\Dt=0.1$. Good agreement is shown on the interior of the domain; boundary effects are visible.}
	\label{Fig1}
\end{figure}

\subsubsection{Singular behaviour at the boundary}
\label{sec:boundary}

We consider here the steady states of the fractional heat equation \eqref{eq:fracHE} on a bounded domain in order to explore the singular behaviour at the boundary. In one dimension, the steady state $\rho\st$ of \eqref{eq:BcFP} with $\beta=0$ satisfies
\begin{align}
	\pxx\prt*{\int_{-R}^{R} \frac{\rho\st(y)}{|x-y|^{\alpha -1}} \dy} = 0,
\end{align}
which, upon considering the boundary conditions, reduces to
\begin{equation}\label{SingIntEq}
	\int_{-R}^{R} \frac{\rho\st(y)}{|x-y|^{\alpha -1}} \dy = C
\end{equation}
for some constant $C$. For $\alpha>1$, the steady profile can be found explicitly:
\begin{align} \label{StSt}
	\rho\st(x)
	 & =
	\cos^{2} \prt*{\frac{\pi (\alpha - 1)}{2}}
	\frac{
	C(R+x)^{\frac{\alpha}{2} - 1}
	}{
	\pi^{2}(R-x)^{1 - \frac{\alpha}{2}}
	}
	\int_{-R}^{R}
	\frac{
		(R-y)^{1-\frac{\alpha}{2}}
		(R+y)^{\frac{\alpha}{2} - 1}
	}{
		y-x
	}
	\dy
	\\&\quad
	+
	\frac{
		A\sin\prt*{\pi (\alpha - 1)}
	}{
		2 \pi(R+x)^{2 - \alpha}
	};
\end{align}
see \cite{EK2000} for details.

\Cref{Fig7} shows the numerical solution ($\alpha=1.5$, $R=50$, $\Dx=0.1$, $\Dt=0.5$) on $\Omega=(-R,R)$ as it tends to the stationary profile \eqref{StSt}. The datum is taken as in the previous section. The explicit steady state is captured by the numerical solution as time grows. Note, however, that we have to run the simulation for a long time before the match is apparent; this is in contrast to the experiment in the next section. The slow convergence may be due to the singular behaviour at the boundary.

\begin{figure}
	\centering

	\begin{subfigure}{0.49\textwidth}
		\centering
		\includegraphics[width=\textwidth]{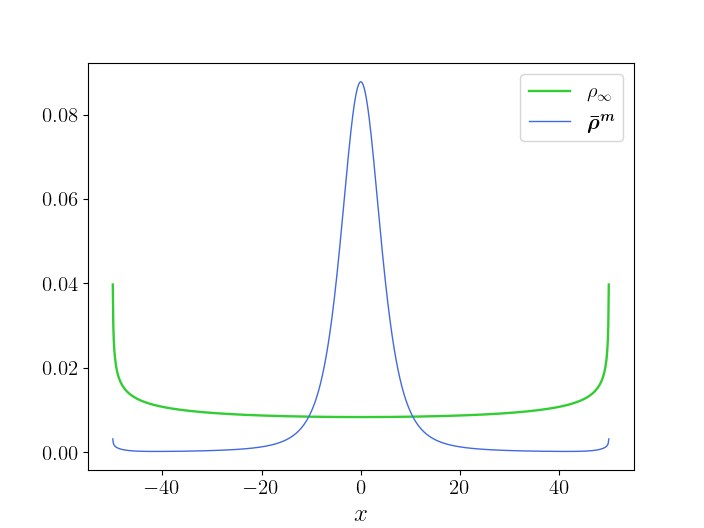}
		\caption{$t=5$}
	\end{subfigure}
	\hfill
	\begin{subfigure}{0.49\textwidth}
		\centering
		\includegraphics[width=\textwidth]{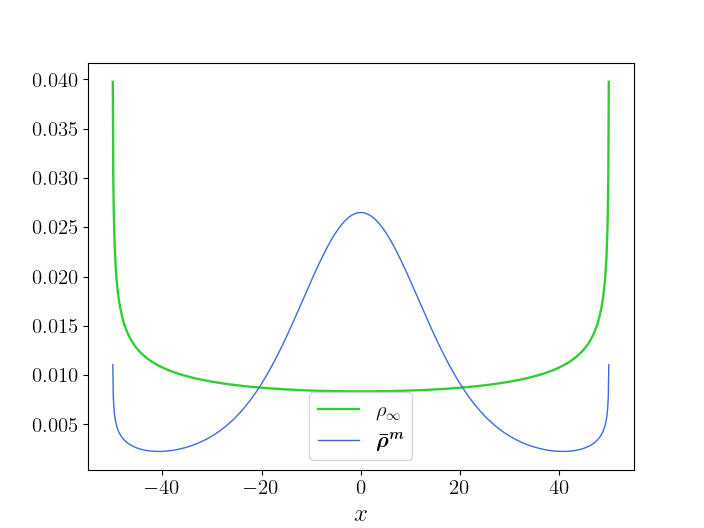}
		\caption{$t=35$}
	\end{subfigure}

	\begin{subfigure}{0.49\textwidth}
		\centering
		\includegraphics[width=\textwidth]{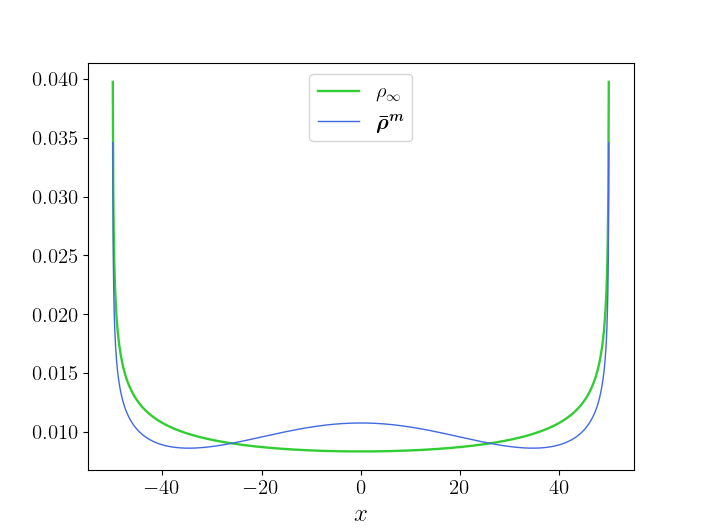}
		\caption{$t=150$}
	\end{subfigure}
	\hfill
	\begin{subfigure}{0.49\textwidth}
		\centering
		\includegraphics[width=\textwidth]{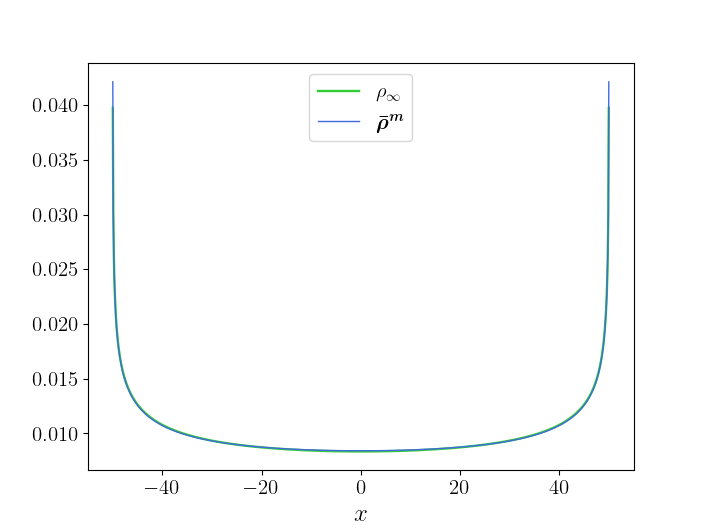}
		\caption{$t=400$}
	\end{subfigure}

	\caption{Fractional heat equation \eqref{eq:fracHE} in one dimension. Numerical solution $\bbrho\m$ and explicit steady state $\rho\st$ on $\Omega=(-R,R)$. Scheme \eqref{eq:scheme1D}, $\alpha = 1.5$, $R=50$, $\Dx=0.1$, $\Dt=0.5$. The numerical solution tends to $\rho\st$.}
	\label{Fig7}
\end{figure}

\subsubsection{Steady states as a function of domain size}\label{sec:steady1D}

We now turn to the L\'{e}vy-Fokker-Planck equation \eqref{eq:fracFP}. First we consider the case $\alpha=1$, where an explicit solution on the whole line is known:
\begin{equation}\label{explicit1D}
	\rho^{\ast}(t,x)
	=
	\frac{1}{\pi}
	\frac{
	e^{t}(e^{t}-1)
	}{
	(1+x^{2})e^{2t} - 2e^{t} + 1
	};
\end{equation}
as $t\rightarrow\infty$, this solution tends to the steady state
\begin{equation}\label{explicitSS1D}
	\rho_{\infty}(x)
	=
	\frac{1}{\pi}
	\frac{1}{1+x^{2}}.
\end{equation}

\Cref{Fig8} shows the the numerical solution ($\alpha=1+\varepsilon$, $R=50$, $\Dx=0.05$, $\Dt=0.01$) on $\Omega=(-R,R)$ compared to the explicit steady state \eqref{explicitSS1D}. The datum for the numerical solution is a uniform distribution with unit mass. Once again, the explicit steady state is captured well by the numerical solution as time grows. Unlike in the previous experiment, this solution approaches the corresponding steady state very rapidly.

\begin{figure}
	\centering

	\begin{subfigure}{0.49\textwidth}
		\includegraphics[width=\textwidth]{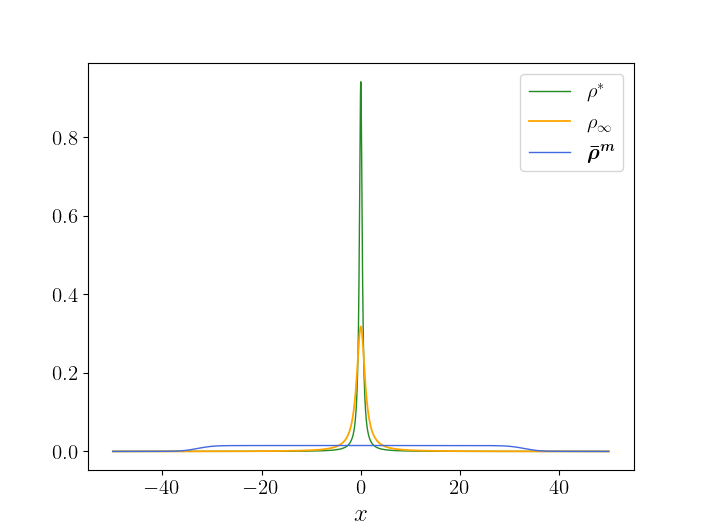}
		\caption{$t=0.4$}
	\end{subfigure}
	\hfill
	\begin{subfigure}{0.49\textwidth}
		\includegraphics[width=\textwidth]{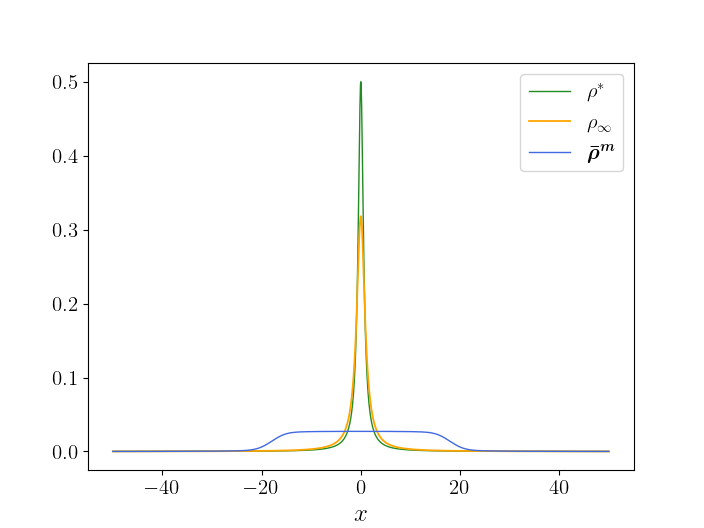}
		\caption{$t=1.0$}
	\end{subfigure}

	\begin{subfigure}{0.49\textwidth}
		\includegraphics[width=\textwidth]{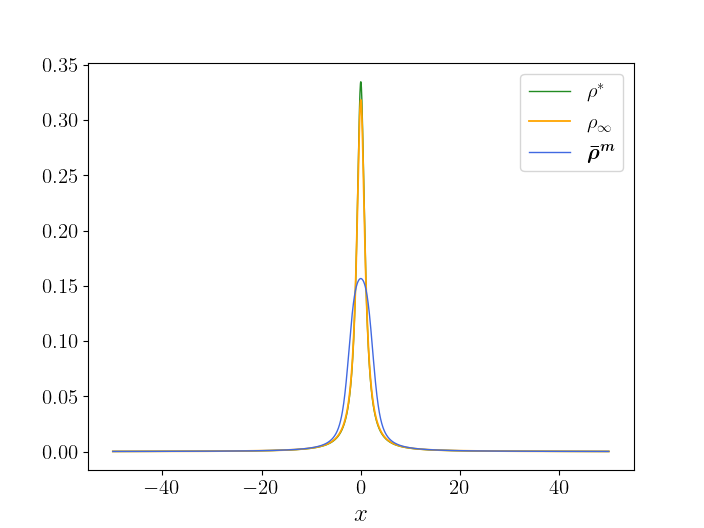}
		\caption{$t=3.0$}
	\end{subfigure}
	\hfill
	\begin{subfigure}{0.49\textwidth}
		\includegraphics[width=\textwidth]{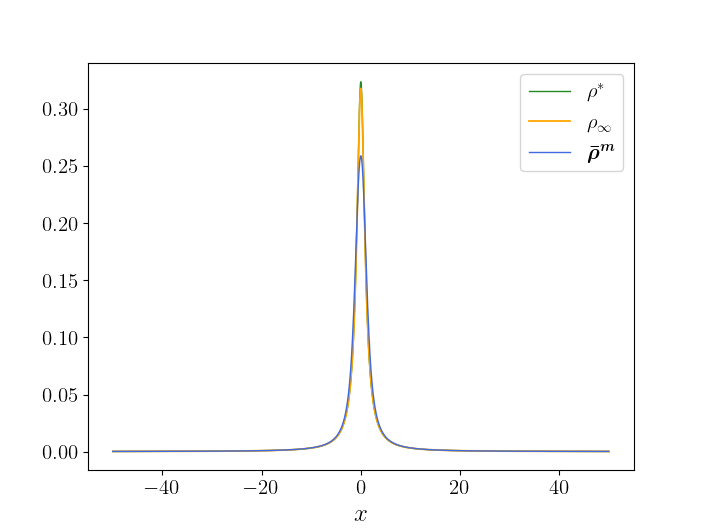}
		\caption{$t=4.0$}
	\end{subfigure}

	\begin{subfigure}{0.7\textwidth}
		\centering
		\includegraphics[width=\textwidth]{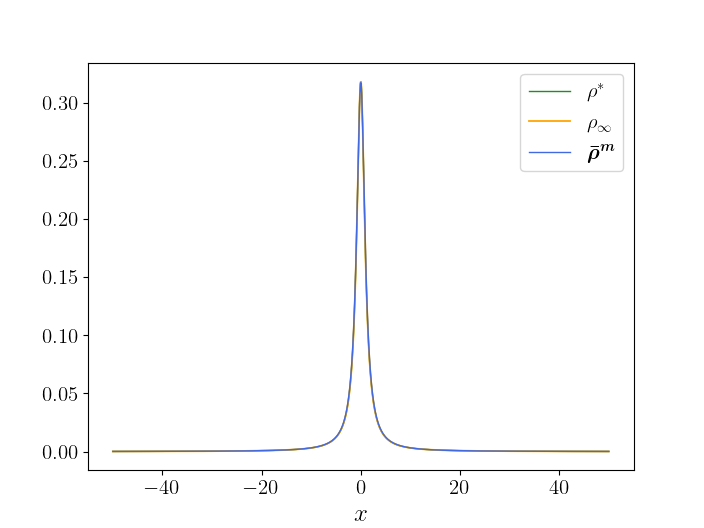}
		\caption{$t=10.0$}
	\end{subfigure}

	\caption{L\'{e}vy-Fokker-Planck equation \eqref{eq:fracFP} in one dimension. Numerical solution $\bbrho\m$, exact solution $\rho^*$, and explicit steady state $\rho\st$. Scheme \eqref{eq:scheme1D}, $\alpha = 1+\varepsilon$, $R=50$, $\Dx=0.05$, $\Dt=0.01$. The numerical solution clearly tends to $\rho\st$.}
	\label{Fig8}
\end{figure}

As was the case with the fractional heat equation, the typical solution of the L\'{e}vy-Fokker-Planck equation is \textit{leptokurtic}, as it has algebraic tails. Thus, the error committed when a whole-space solution is restricted to a bounded domain is not exponentially small, even if the presence of the Fokker-Planck term prevents singularities from developing at the boundary. We therefore expect that the steady state in a bounded domain will differ from \eqref{explicitSS1D} by a non-trivial amount.

\Cref{Fig12} shows the $\Lone\prt{\Omega}$ distance between the numerical steady state of the L\'{e}vy-Fokker-Planck equation ($\alpha=1+\varepsilon$, $\Dx=2R/2^{12}$, $\Dt=0.1$) on $\Omega=\prt{-R,R}$ for various values of $R$, and the explicit steady state \eqref{explicitSS1D}. As expected, the error decreases as $R$ tends to infinity, though the decay does not follow an obvious pattern.

We show the relative entropy of our numerical solution (for $\alpha = 1+\varepsilon$) with respect to the equilibrium $\rho_{\infty}$ in \Cref{Fig16}. The results show good agreement with the exponential trend predicted by \cite{gentil2008levy}, using two most common entropy functions: $\Phi(x) = (x -1)^{2}$ and $\Phi(x) = x(\log x - 1) + 1$.

\begin{figure}
	\centering
	\includegraphics[width=0.9\textwidth]{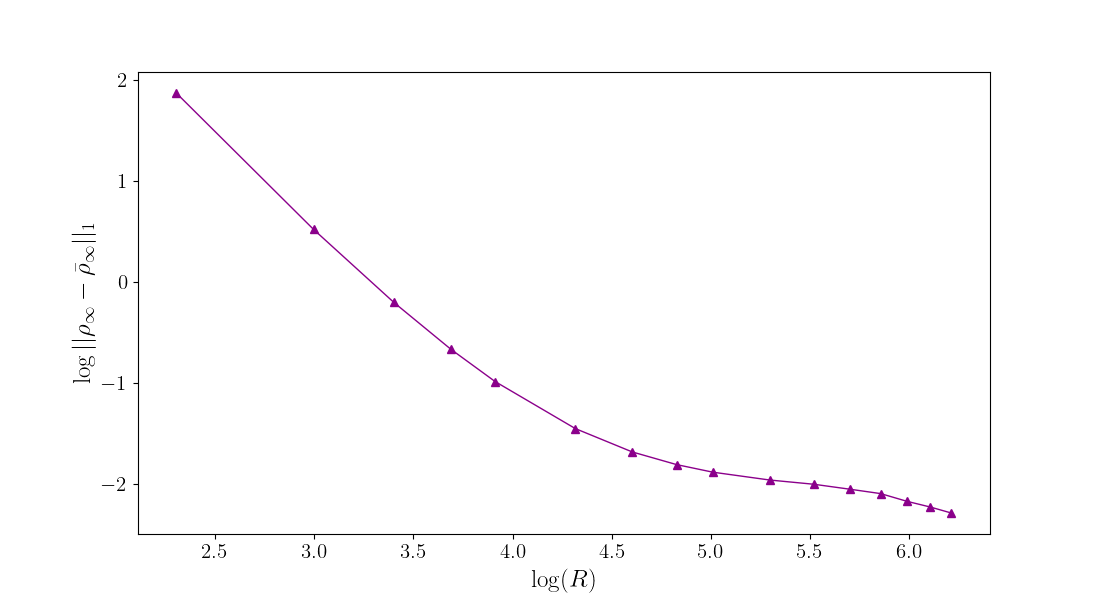}
	\caption{L\'{e}vy-Fokker-Planck equation \eqref{eq:fracFP} in one dimension. $\Lone\prt{\Omega}$ distance between the numerical steady state with $\alpha=1+\varepsilon$ on $\Omega=\prt{-R,R}$ and the explicit steady state with $\alpha=1$. Scheme \eqref{eq:scheme1D}, $\Dx=2R/2^{12}$, $\Dt=0.1$. The mismatch decreases as $R$ increases.}
	\label{Fig12}
\end{figure}

A similar analysis can be performed when $\alpha/2=1$. The L\'{e}vy-Fokker-Planck equation reduces to the classical Fokker-Planck equation
\begin{align}
	\pt\rho = \laplace\rho + \div\prt{x\rho},
\end{align}
whose unique, asymptotically stable steady state is
\begin{align}\label{eq:FokkerPlanckSS}
	\rho^{\infty}(x) = \frac{1}{(2 \pi)^{\frac{d}{2}}} e^{-\frac{|x|^{2}}{2}}
\end{align}
in dimension $d$ (for solutions with unit mass on the whole space). If we let $\frac{\alpha}{2} = 0.99$, the steady state of our numerical scheme should be close to this one, and the agreement should improve as the domain grows.

\Cref{Fig15} shows the $\Lone\prt{\Omega}$ distance between the numerical steady state of the L\'{e}vy-Fokker-Planck equation ($\alpha/2=0.99$, $\Dx=2R/2^{12}$, $\Dt=0.1$) on $\Omega=\prt{-R,R}$ for various values of $R$ and the explicit steady state \eqref{eq:FokkerPlanckSS}. Once again, the error decreases as $R$ tends to infinity, as expected.

\Cref{fig:LFP1DSteadyStates} shows the numerical steady states of the L\'{e}vy-Fokker-Planck equation \eqref{eq:fracFP} ($R=50$, $\Delta x = 0.05$, $\Delta t = 0.01$) on $\Omega=(-R,R)$ for various fractional orders $\alpha \in (1,2)$. We recover symmetric distributions with algebraic tails that become thicker as $\alpha$ decreases. We compare the tails of our numerical results with the expected behaviour predicted in \cite{blumenthal1960TheoremsStableProcesses}, given by $\rho(t,x) \asymp \min\{ 1 , |x|^{-\alpha-d} \}$.

\begin{figure}
	\centering
	\includegraphics[width=0.9\textwidth]{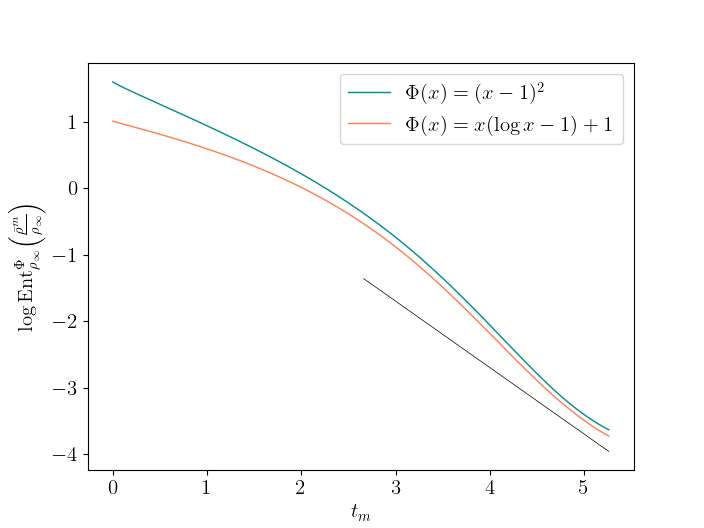}
	\caption{L\'{e}vy-Fokker-Planck equation \eqref{eq:fracFP} in one dimension. Dissipation of the relative entropy with $\alpha=1 + \varepsilon$ with respect to the equilibrium $\rho_{\infty}$. Scheme \eqref{eq:scheme1D}, $R=50$, $\Dx=0.05$, $\Dt=0.01$. Black reference line has slope one. }
	\label{Fig16}
\end{figure}

\begin{figure}
	\centering
	\includegraphics[width=0.9\textwidth]{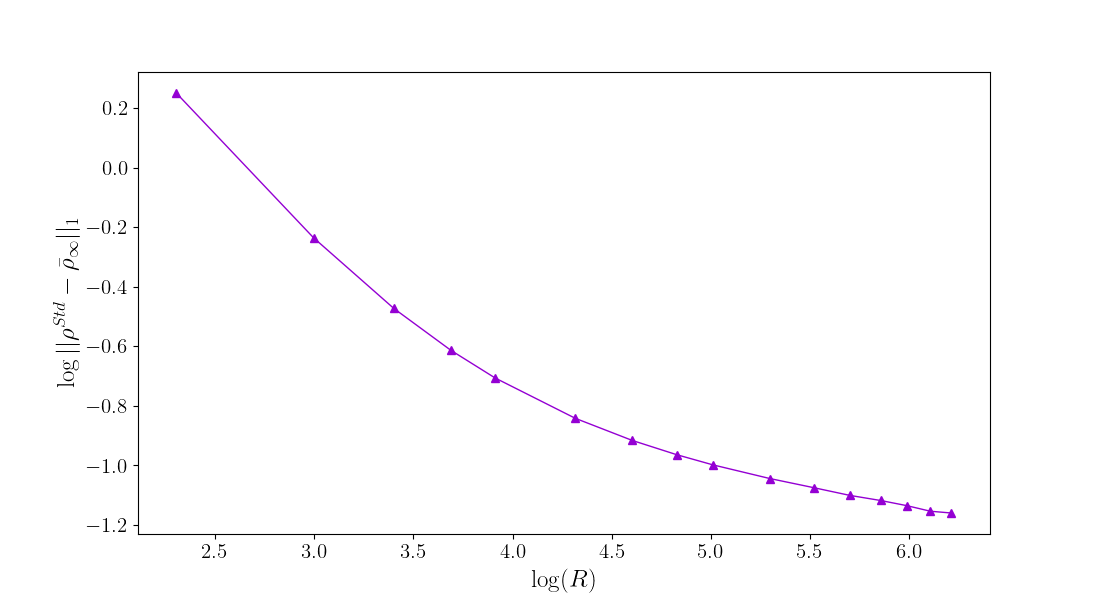}
	\caption{L\'{e}vy-Fokker-Planck equation \eqref{eq:fracFP} in one dimension. $\Lone\prt{\Omega}$ distance between the numerical steady state with $\alpha/2=0.99$ on $\Omega=\prt{-R,R}$ and the explicit steady state with $\alpha=2$. Scheme \eqref{eq:scheme1D}, $\Dx=2R/2^{12}$, $\Dt=0.1$. The distance decreases as $R$ increases.}
	\label{Fig15}
\end{figure}

\begin{figure}
	\centering
	\includegraphics[width=0.9\textwidth]{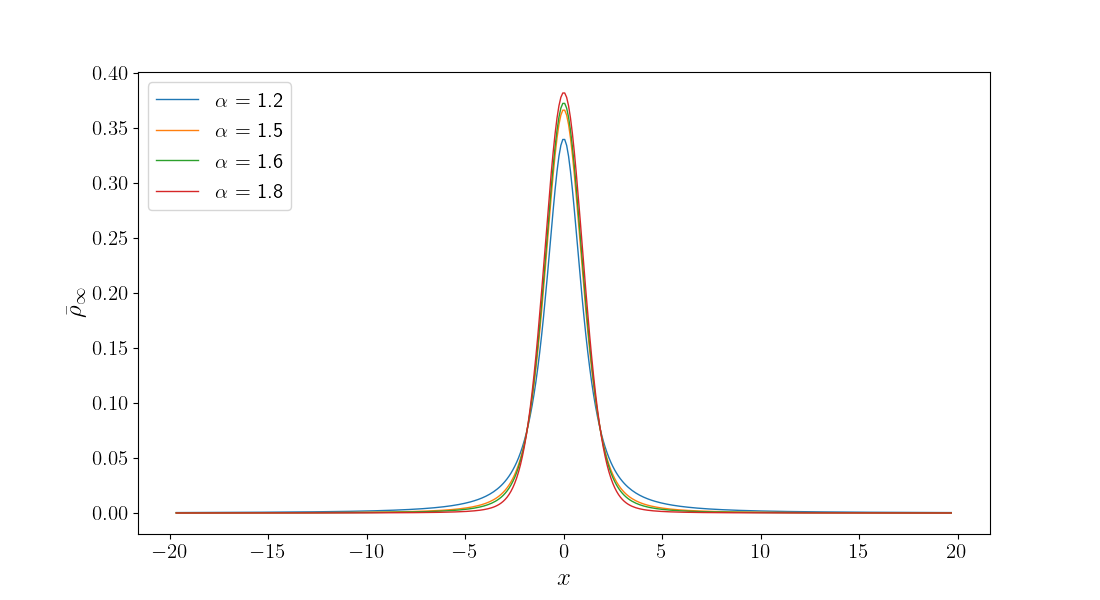}
	\includegraphics[width = 0.9\textwidth]{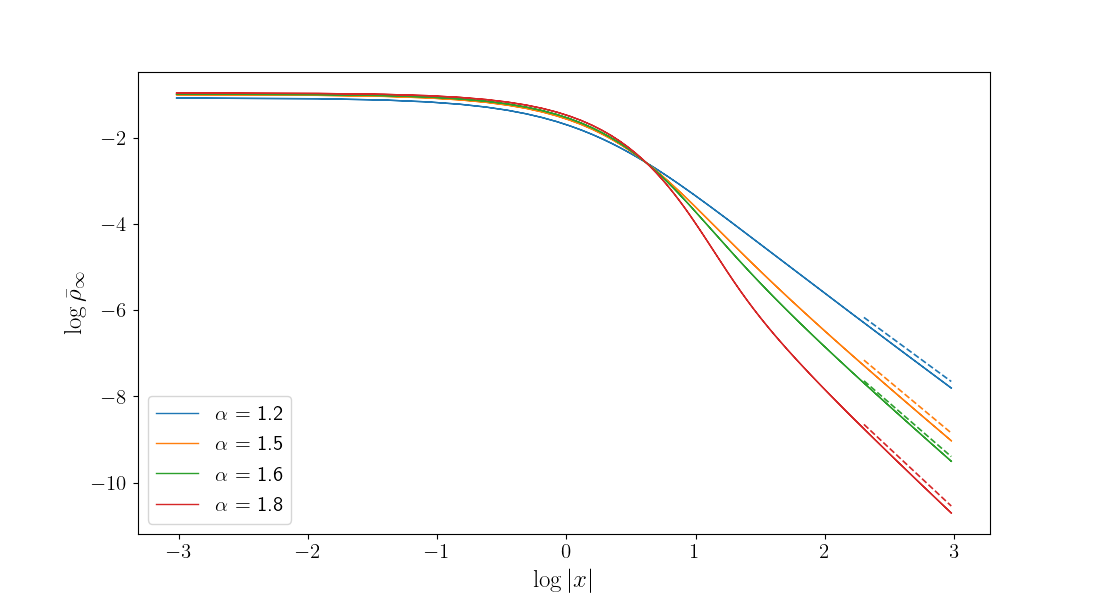}
	\caption{L\'{e}vy-Fokker-Planck equation \eqref{eq:fracFP} in one dimension. Numerical steady state for varying fractional order $\alpha\in(1,2)$. Scheme \eqref{eq:scheme1D}, $R=50$, $\Dx=0.05$, $\Dt=0.01$. \textbf{Top:} profile at the centre of the domain. \textbf{Bottom:} detail of the algebraic tails, compared with the predicted asymptotic behaviour $|x|^{-\alpha-d}$ (dashed).}
	\label{fig:LFP1DSteadyStates}
\end{figure}

\subsubsection{Convergence of steady states}

To conclude, we verify the order of convergence of the scheme. We fix the domain and compute the steady state of the L\'{e}vy-Fokker-Planck equation \eqref{eq:fracFP} as in the previous section, for various values of $\alpha$. We compute the steady states on a sequence of refining meshes, and study their convergence. Since the analytical steady state is not known explicitly, we shall monitor the error between numerical steady states, and show that this decays with the mesh size.

\Cref{Fig2} shows the $\Lone\prt{\Omega}$ and $\Ltwo\prt{\Omega}$ distance between successive numerical steady states ($R=50$, $\Dt=\Dx=2R/ 2^{n}$ for $5\leq n\leq 10$) computed with scheme \eqref{eq:scheme1D} on $\Omega=(-R,R)$. The scheme is first-order accurate for all fractional orders $\alpha\in\prt{1,2}$.

\begin{figure}
	\centering

	\begin{subfigure}{0.49\textwidth}			\includegraphics[width=\textwidth]{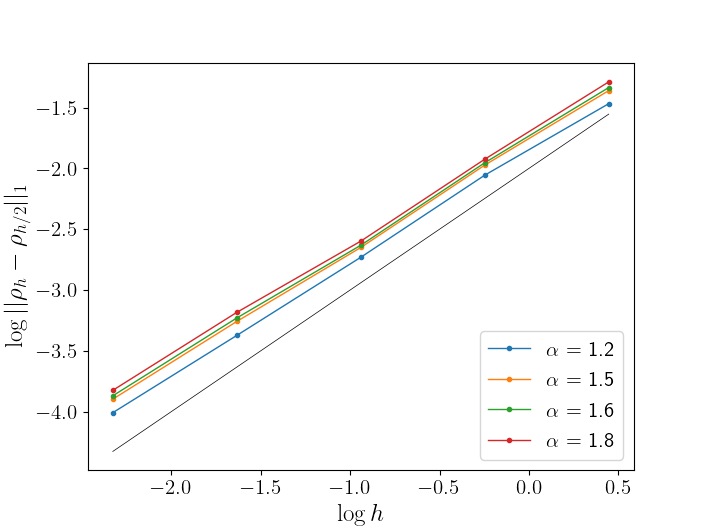}
		\caption{$L^{1}(\Omega)$ norm}
	\end{subfigure}
	\hfill
	\begin{subfigure}{0.49\textwidth}
		\includegraphics[width=\textwidth]{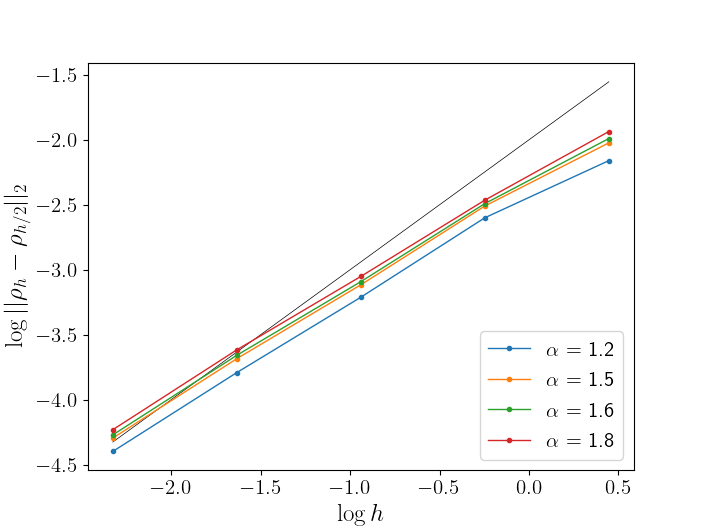}
		\caption{$L^{2}(\Omega)$ norm}
	\end{subfigure}

	\caption{L\'{e}vy-Fokker-Planck equation \eqref{eq:fracFP} in one dimension. Convergence of scheme \eqref{eq:scheme1D} for varying fractional order $\alpha$. $R=50$, $\Dx=2R/2^{n}$ for $5\leq n\leq 10$, $\Dt = \Dx$. Black reference line has slope one.}
	\label{Fig2}
\end{figure}

\Cref{Fig3} performs the same analysis on scheme \eqref{eq:scheme1D} with the second-order flux \eqref{eq:minmodFlux} (viz. \cref{rm:higherOrder}), letting $\Dt = \Dx^2$ instead. The scheme is second-order accurate for all fractional orders $\alpha\in\prt{1,2}$.

\begin{figure}
	\centering

	\begin{subfigure}{0.49\textwidth}
		\includegraphics[width=\textwidth]{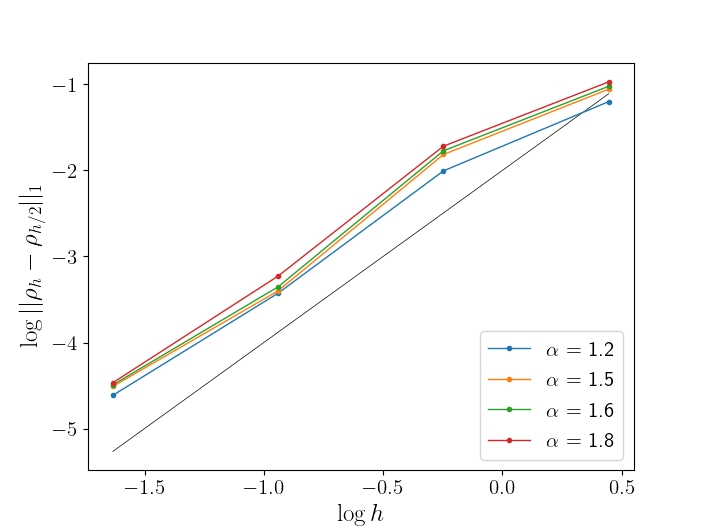}
		\caption{$L^{1}(\Omega)$ norm}
	\end{subfigure}
	\hfill
	\begin{subfigure}{0.49\textwidth}
		\includegraphics[width=\textwidth]{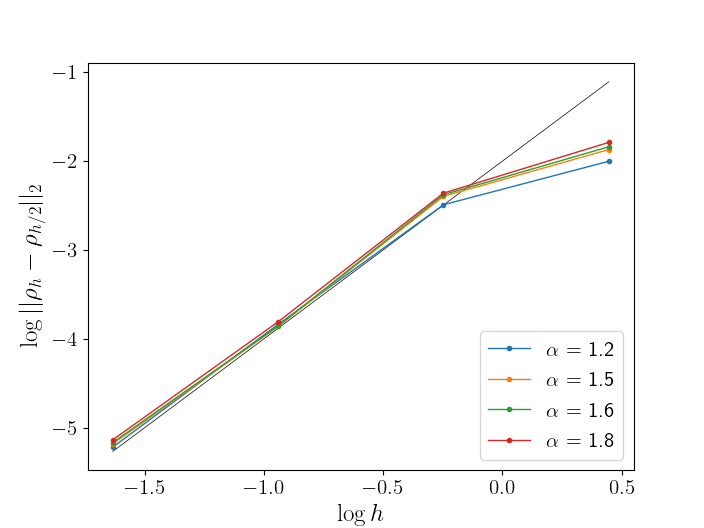}
		\caption{$L^{2}(\Omega)$ norm}
	\end{subfigure}

	\caption{L\'{e}vy-Fokker-Planck equation \eqref{eq:fracFP} in one dimension. Convergence of scheme \eqref{eq:scheme1D} with second order flux \eqref{eq:minmodFlux} (viz. \cref{rm:higherOrder}) for varying fractional order $\alpha$. $R=50$, $\Dx=2R/2^{n}$ for$5\leq n\leq 10$, $\Dt=\Dx$. Black reference line has slope two.}
	\label{Fig3}
\end{figure}
 \subsection{Two dimensions}

\subsubsection{Steady states as a function of domain size}\label{sec:steady2D}

We begin our two-dimensional experiments by verifying the behaviour of the dimensionally split scheme. \Cref{fig:LFP2DSteadyStates} shows the numerical steady states of the L\'{e}vy-Fokker-Planck equation \eqref{eq:fracFP} ($R=20$, $\Delta x = \Delta y = 0.15$, $\Delta t = 0.2$) on $\Omega=(-R,R)^2$ for various fractional orders. We recover radially symmetric distributions with algebraic tails that become thicker as $\alpha$ decreases. We compare the tails of our numerical results with the expected behaviour predicted in \cite{blumenthal1960TheoremsStableProcesses}, which is given by
$\rho(t,x) \asymp \min\{ 1 , |x|^{-\alpha-d} \}$.

\begin{figure}
	\centering
	\includegraphics[width=0.9\textwidth]{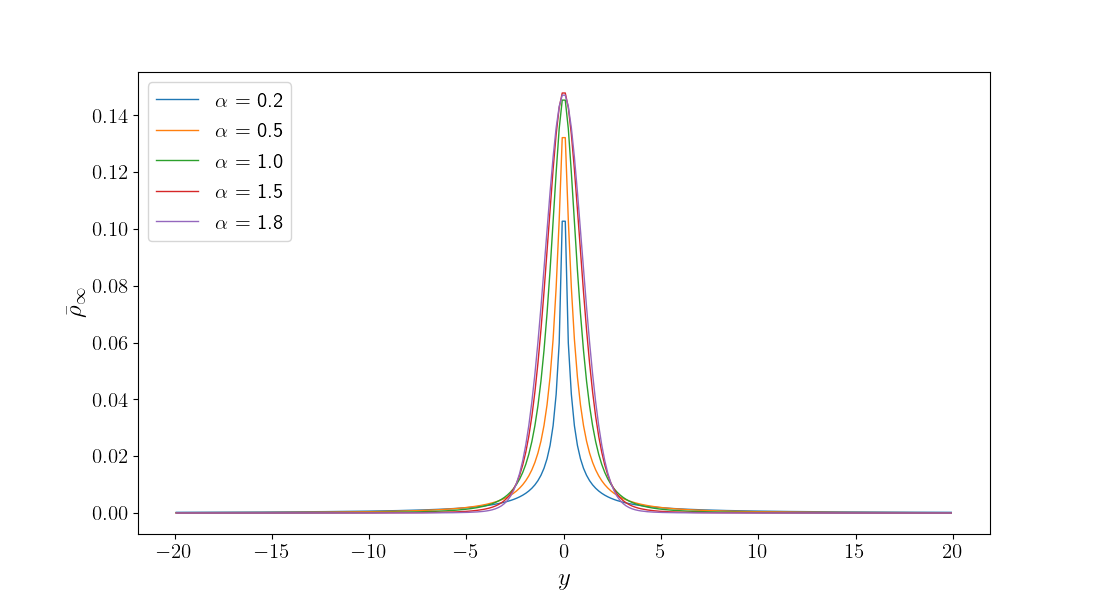}
	\includegraphics[width = 0.9\textwidth]{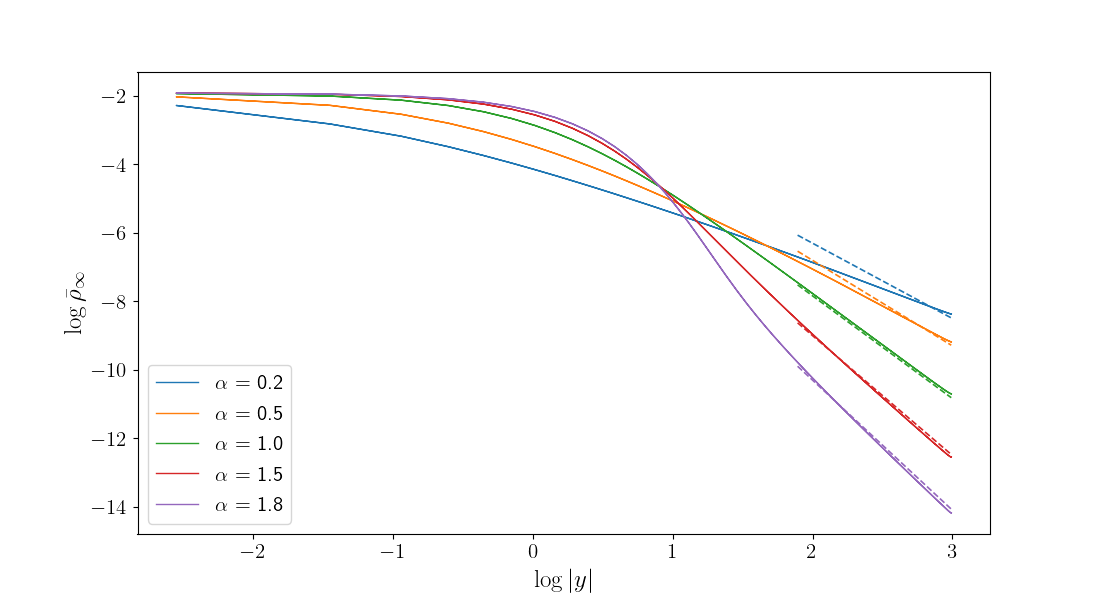}
	\caption{L\'{e}vy-Fokker-Planck equation \eqref{eq:fracFP} in two dimensions. Numerical steady state for varying fractional order $\alpha\in(0,2)$. Scheme \eqref{eq:scheme2D} with splitting (viz. \cref{sec:splitting}), $R=20$, $\Dx=\Dy=0.15$, $\Dt=0.2$. \textbf{Top:} central section. \textbf{Bottom:} detail of the algebraic tails, compared with the dotted lines of the predicted long time asymptotic behaviour $|x|^{-\alpha-d}$.}
	\label{fig:LFP2DSteadyStates}
\end{figure}

As in the one-dimensional case, an explicit solution to the L\'{e}vy-Fokker-Planck equation on the whole space is known for $\alpha=1$. The solution is found from the self-similar solution to the fractional heat equation \cite{de2022stationary} through the change of variables proposed in \cite{biler2003generalised}, just as a solution to the classical Fokker-Planck equation can be derived from a solution to the heat equation. The solution in question is given by
\begin{equation}\label{explicit2D}
	\rho^{\ast}(t, x, y)
	=
	\frac{1}{2\pi}
	\frac{
		e^{2t}(e^{t}-1)
	}{
		\prt*{(1+x^2+y^2)e^{2t} - 2e^{t} + 1}^{\frac{3}{2}}
	},
\end{equation}
which tends to the steady state
\begin{equation}\label{explicitSS2D}
	\rho_{\infty}(x, y)
	=
	\frac{1}{2\pi}
	\frac{1}{\prt*{1+x^2+y^2}^{\frac{3}{2}}}.
\end{equation}

\Cref{Fig9} shows the the numerical solution ($\alpha=1$, $R=20$, $\Dx=\Dy=0.08$, $\Dt=0.1$) on $\Omega=(-R,R)^2$ compared to the explicit steady state \eqref{explicitSS2D}. The datum for the numerical solution is a uniform distribution with unit mass.

\begin{figure}
	\centering

	\begin{subfigure}{0.49\textwidth}
		\includegraphics[
			trim=100pt 30pt 70pt 60pt, clip,
			width=\textwidth]{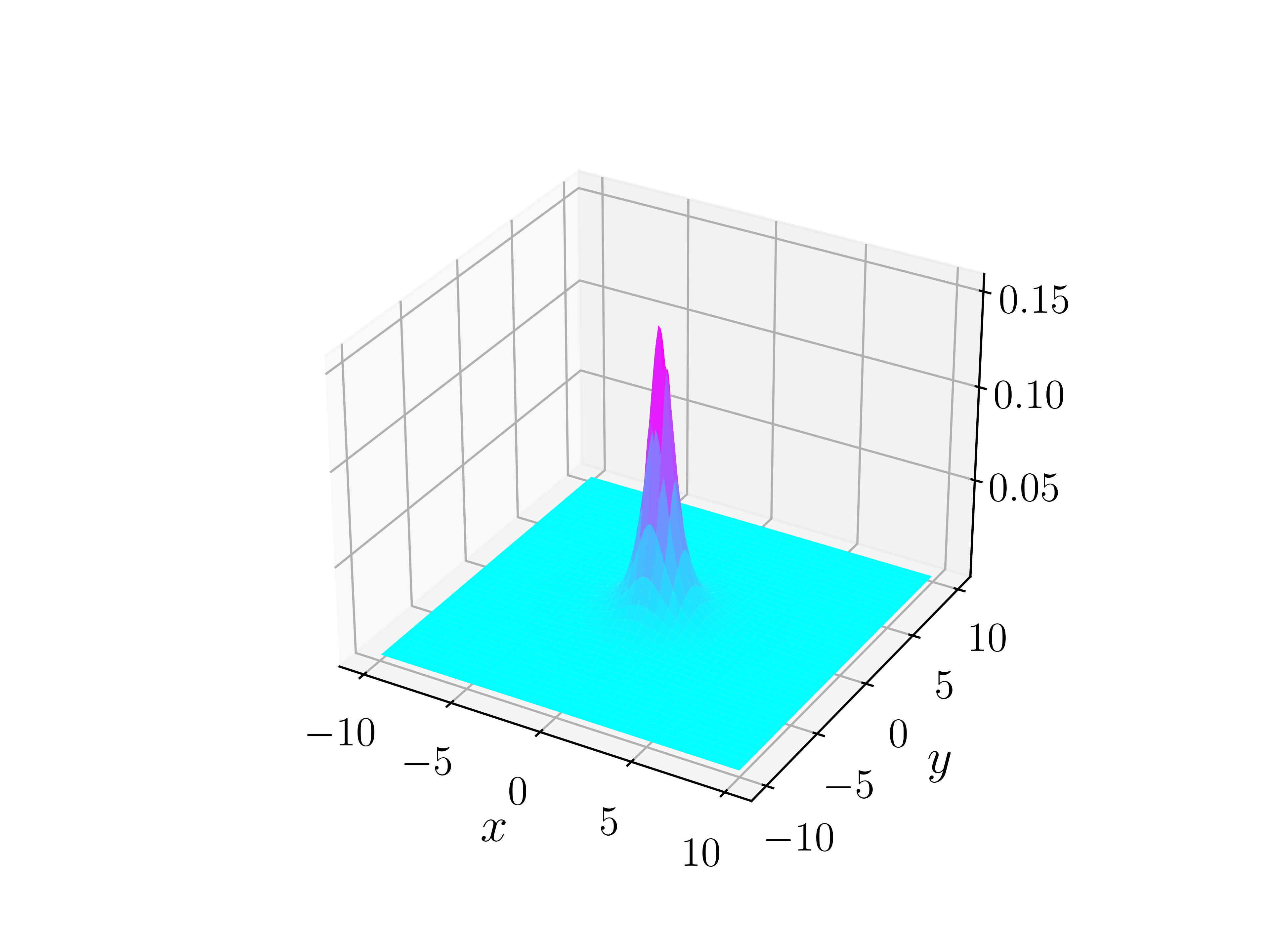}
		\caption{Numerical steady state}
	\end{subfigure}
	\hfill
	\begin{subfigure}{0.49\textwidth}
		\includegraphics[width=\textwidth]{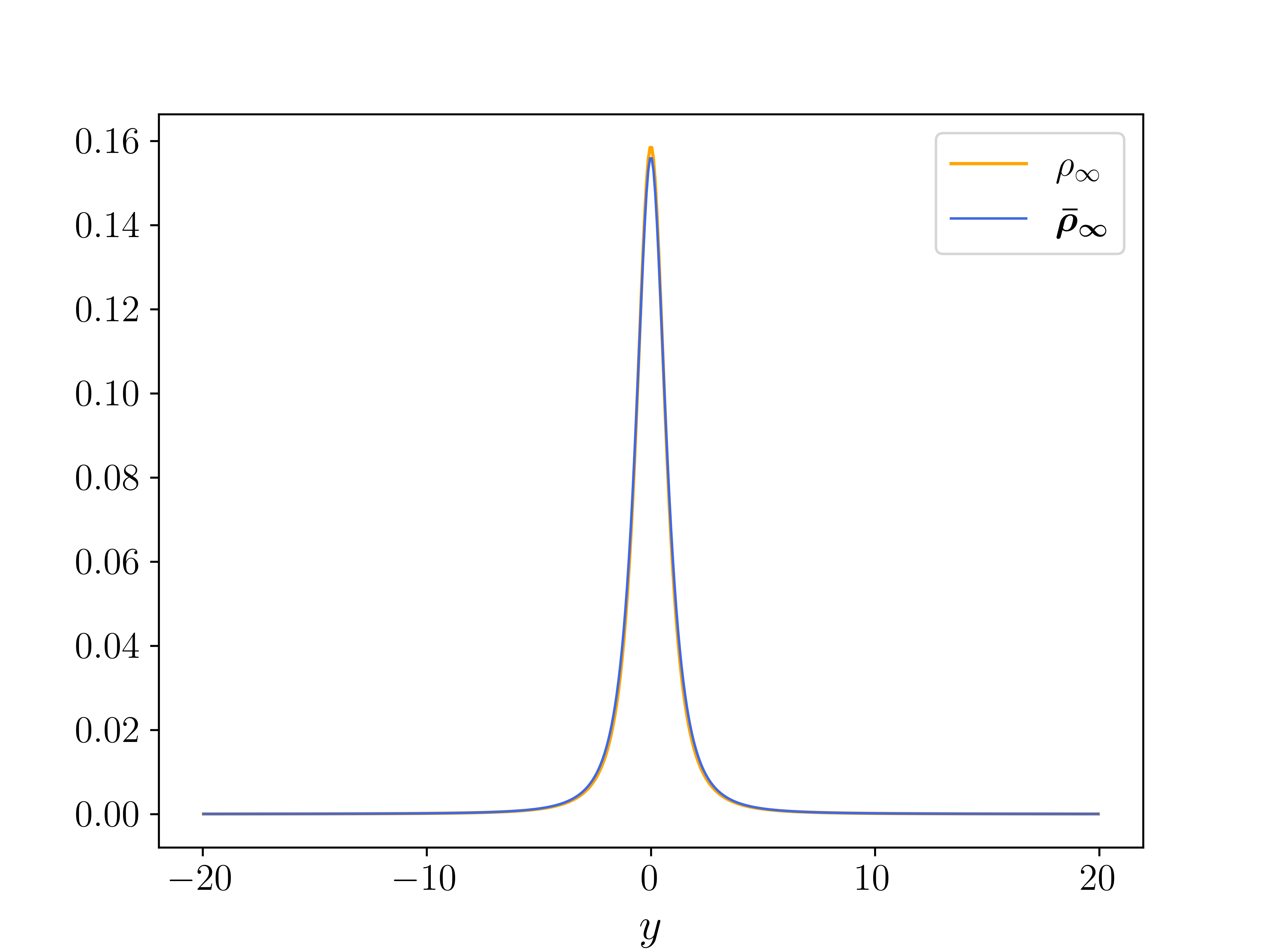}
		\caption{Central section}
	\end{subfigure}

	\caption{L\'{e}vy-Fokker-Planck equation \eqref{eq:fracFP} in two dimensions. Numerical steady state and explicit steady state. Scheme \eqref{eq:scheme2D} with splitting (viz. \cref{sec:splitting}), $\alpha = 1+\varepsilon$, $R=20$, $\Dx=\Dy=0.08$, $\Dt=0.1$.}
	\label{Fig9}
\end{figure}

We now study the convergence of the numerical steady state to the profile \eqref{explicitSS2D} as the size of the domain grows. Unlike the one-dimensional test, the two-dimensional analysis can be performed setting $\alpha=1$ exactly. \Cref{Fig10} shows the $\Lone\prt{\Omega}$ distance between the numerical steady state of the L\'{e}vy-Fokker-Planck equation ($\alpha=1$, $\Dx=\Dy=2R/2^{8}$, $\Dt=0.1$) on $\Omega=\prt{-R,R}^2$, for various values of $R$, and the explicit steady state \eqref{explicitSS2D}. As in the one-dimensional case, the error decreases as $R$ tends to infinity.

\begin{figure}
	\centering
	\includegraphics[width=0.9\textwidth]{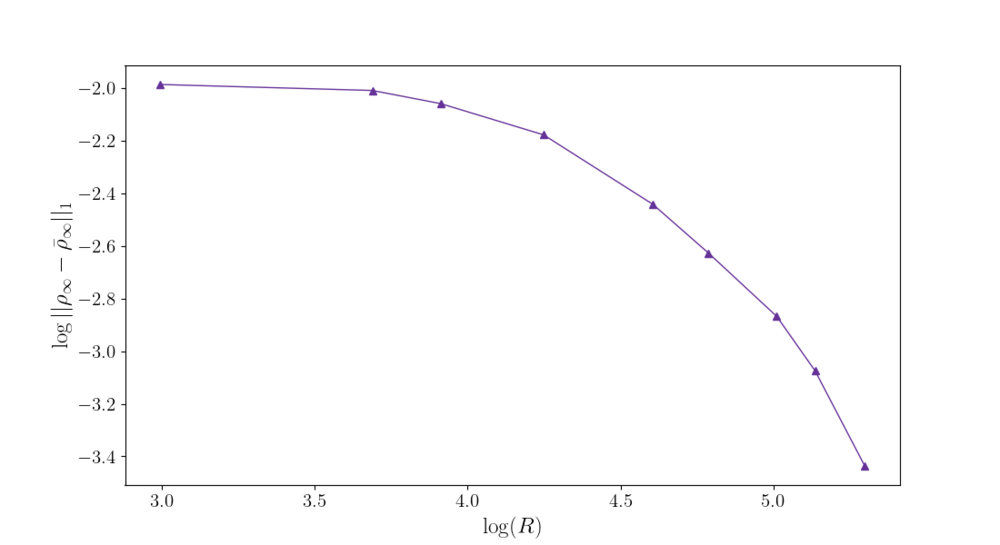}
	\caption{L\'{e}vy-Fokker-Planck equation \eqref{eq:fracFP} in two dimensions. $\Lone\prt{\Omega}$ distance between the numerical steady state on $\Omega=\prt{-R,R}^2$ and the explicit steady state. Scheme \eqref{eq:scheme2D} with splitting (viz. \cref{sec:splitting}), $\alpha=1$, $\Dx=2R/2^{8}$, $\Dt=0.1$. The mismatch decreases as $R$ increases.}
	\label{Fig10}
\end{figure}

\subsubsection{Convergence of steady states}

We verify the order of convergence of the dimensionally split scheme. As in one dimension, we fix the domain size and compute the steady state of the L\'{e}vy-Fokker-Planck equation \eqref{eq:fracFP} for various values of $\alpha$. \Cref{Fig5} shows the $\Lone\prt{\Omega}$ and $\Ltwo\prt{\Omega}$ distance between numerical steady states ($R=20$, $\Dt=\Dx=\Dy=2R/ 2^{n}$ for $5\leq n\leq 8$) computed with scheme \eqref{eq:scheme2D} on $\Omega=(-R,R)^2$ as the mesh size is halved. The order of the scheme appears slightly less than one; this might be a consequence of the dimensional splitting. Noticeably, the convergence is initially very slow when the fractional order is close to zero.

\begin{figure}
	\centering

	\begin{subfigure}{0.49\textwidth}
		\includegraphics[width=\textwidth]{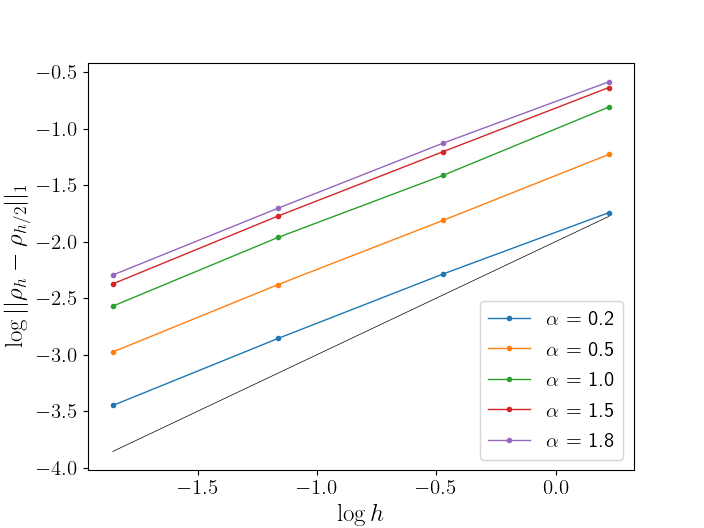}
		\caption{$L^{1}(\Omega)$ norm}
	\end{subfigure}
	\hfill
	\begin{subfigure}{0.49\textwidth}
		\includegraphics[width=\textwidth]{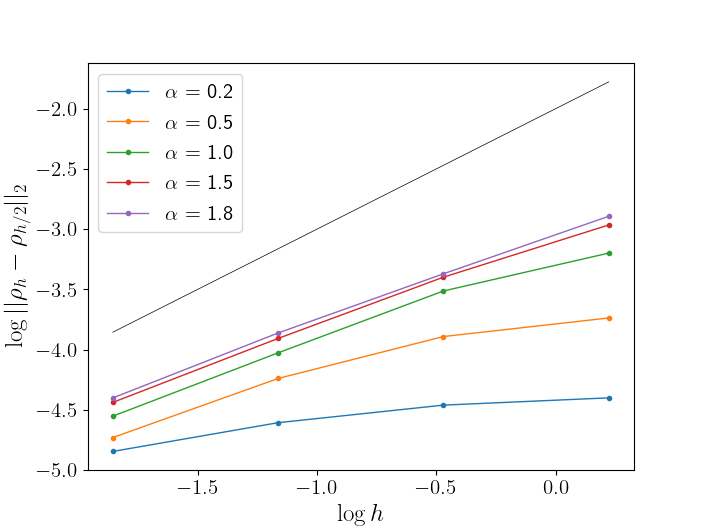}
		\caption{$L^{2}(\Omega)$ norm}
	\end{subfigure}

	\caption{L\'{e}vy-Fokker-Planck equation \eqref{eq:fracFP} in two dimensions. Convergence of scheme \eqref{eq:scheme2D} with splitting (viz. \cref{sec:splitting}) for varying fractional orders $\alpha$. $R=20$, $\Dx=\Dy=2R/2^{n}$ for $n=5, \dots, 8$, $\Dt=\Dx$. Black reference line has slope one.}
	\label{Fig5}
\end{figure}

\subsubsection{Long-time asymptotics}

To conclude, we study the rate of convergence of the numerical solution of the L\'{e}vy-Fokker-Planck equation \eqref{eq:fracFP} to the corresponding steady states. \Cref{Fig6} shows the $\Lone(\Omega)$ and $\Ltwo(\Omega)$ distances of the numerical solution ($R=20$, $\Delta x = \Delta y = 0.15$, $\Delta t = 0.08$) on $\Omega=(-R,R)^2$ for various fractional orders to their asymptotic steady states as a function of time. Perhaps due to the highly symmetric initial data, the numerical solutions show an improved rate of convergence ($e^{-2t}$) towards the steady state with respect to the result of \cite{gentil2008levy} ($e^{-\alpha t}$). This acceleration phenomena due to symmetry of the datum is well-documented, as it has been observed also in the classical Fokker-Planck setting \cite{BARTIER201176}, as well as in the porous medium equation \cite{carrillo2007strict}.

\begin{figure}
	\centering

	\begin{subfigure}{0.49\textwidth}
		\includegraphics[width=\textwidth]{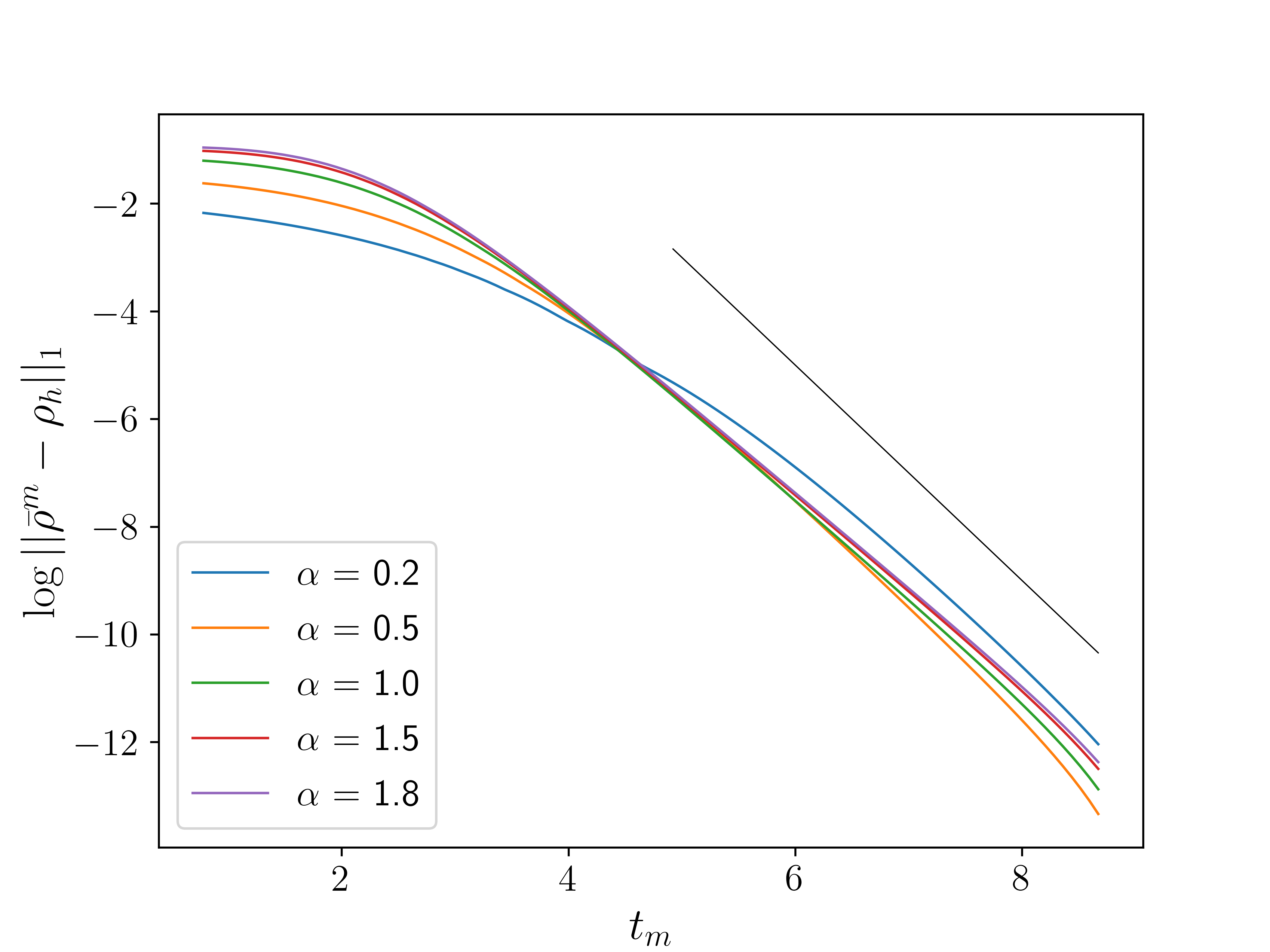}
		\caption{$L^{1}(\Omega)$ norm}
	\end{subfigure}
	\hfill
	\begin{subfigure}{0.49\textwidth}
		\includegraphics[width=\textwidth]{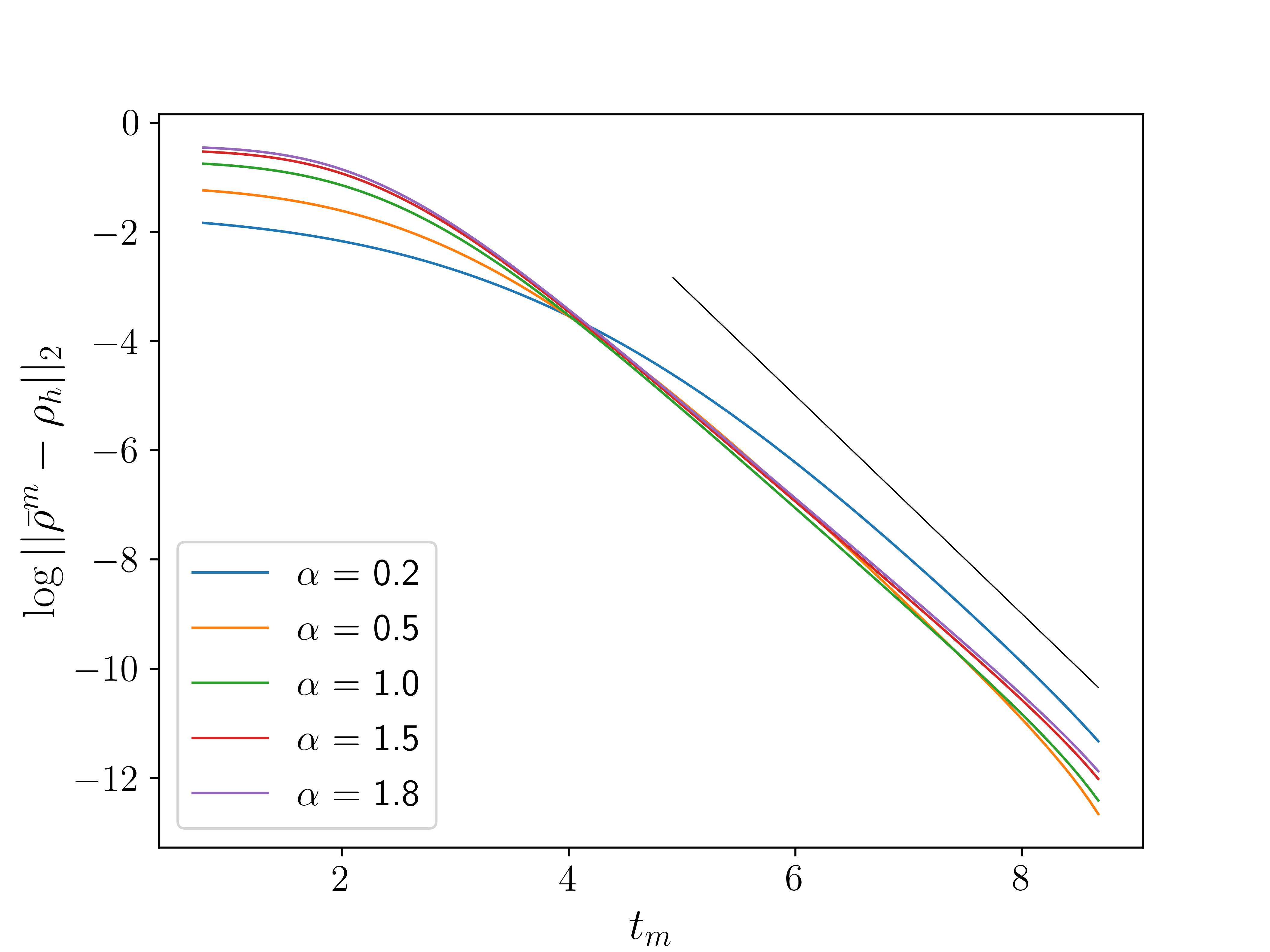}
		\caption{$L^{2}(\Omega)$ norm}
	\end{subfigure}

	\caption{L\'{e}vy-Fokker-Planck equation \eqref{eq:fracFP} in two dimensions. Convergence to steady state of numerical solutions for varying fractional order. Scheme \eqref{eq:scheme2D} with splitting (viz. \cref{sec:splitting}), $R=20$, $\Delta x = \Delta y = 0.15$, $\Delta t = 0.08$. Black reference line has slope two.}
	\label{Fig6}
\end{figure}
  
\section*{Acknowledgements}

This work was supported by the Advanced Grant Nonlocal-CPD (Nonlocal PDEs for Complex Particle Dynamics: Phase Transitions, Patterns and Synchronization) of the European Research Council Executive Agency (ERC) under the European Union’s Horizon 2020 research and innovation programme (grant agreement No. 883363). RB and JAC were also supported by the EPSRC grant numbers EP/T022132/1. JAC was also partially supported by EP/V051121/1. DGC was partially supported by RYC2022-037317-I and PID2021-127105NB-I00 from the Spanish Government. Part of this work was done during the visit of SF as master student from University of Trento by the Erasmus+ programme. 
\renewcommand{\appendixname}{Appendix: A Second Order Discretisation}
\addappendix
\label{sec:appendixMinmod}

The discretisation of the advection terms described in \cref{sec:numericalSchemes} is only accurate to first order. However, the discretisation of the fractional diffusion term is second-order accurate, as discussed in \cref{rm:higherOrder}. In order to verify this, the validation tests of \cref{sec:numericalExperiments} employ a higher order scheme for the advection part. The discretisation of choice is classical: upwind with a \textit{minmod} limiter \cite{LeVeque1990,LeVeque2002}, which has been used successfully for generalised Fokker-Planck equations \cite{CCH2015,BCH2020}. For the sake of self-consistency, we detail here the one-dimensional discretisation.

The definition of the diffusive flux $F\dif\ih(t)$ given in \eqref{eq:1DDiffusionFlux} is not modified. Similarly, the advective velocity $v\ih$ is kept as given in \eqref{eq:1DAdvectionVelocity}. The only alteration takes place in the advective flux $F\ad\ih(t)$; the first-order upwind formula \eqref{eq:1DAdvectionFlux} is replaced by
\begin{align}\label{eq:minmodFlux}
	F\ad\ih(t)
	 & =
	\bar\rho\E\i(t)\pos{v\ih}
	+ \bar\rho\W\ip(t)\neg{v\ih}.
\end{align}
These \textit{east} and \textit{west} values are computed from a piecewise linear reconstruction:
\begin{align}\label{LinearApprox}
	\bar\rho\E\i(t)
	=
	\bar\rho\i(t) + \frac{\Dx}{2} \d\bar\rho\i(t)
	,\quad
	\bar\rho\W\i(t)
	=
	\bar\rho\i(t) - \frac{\Dx}{2} \d\bar\rho\i(t).
\end{align}
The discrete gradient $\d\bar\rho\i$ is defined as
\begin{align}
	\d\bar\rho\i(t)
	=
	\minmod\prt*{
		\frac{\bar\rho\i - \bar\rho\im}{\Dx},
		\frac{\bar\rho\ip - \bar\rho\im}{2\Dx},
		\frac{\bar\rho\ip - \bar\rho\i}{\Dx}
	},
\end{align}
where
\begin{align}
	\minmod(a,b,c) \coloneqq
	\begin{cases}
		\min \set{a,b,c} & \text{if }a,b,c>0, \\
		\max \set{a,b,c} & \text{if }a,b,c<0, \\
		0                & \text{otherwise}.
	\end{cases}
\end{align}
 
\FloatBarrier

\bibliographystyle{abbrv}
\bibliography{./fractional-laplacian.bib}

\begin{thebibliography}{10}

\bibitem{AbatangeloGomezCastroVazquez2022}
N.~Abatangelo, D.~G\'omez-Castro, and J.~L. V\'azquez.
\newblock Singular boundary behaviour and large solutions for fractional
  elliptic equations.
\newblock {\em Journal of the London Mathematical Society}, 107:568--615, 2023.

\bibitem{abatangelo2023singular}
N.~Abatangelo, D.~G{\'o}mez-Castro, and J.~L. V{\'a}zquez.
\newblock Singular boundary behaviour and large solutions for fractional
  elliptic equations.
\newblock {\em Journal of the London Mathematical Society}, 107(2):568--615,
  2023.

\bibitem{AB2017}
G.~Acosta and J.~P. Borthagaray.
\newblock A fractional {L}aplace equation: Regularity of solutions and finite
  element approximations.
\newblock {\em {SIAM} Journal on Numerical Analysis}, 55(2):472--495, jan 2017.

\bibitem{AG2018}
M.~Ainsworth and C.~Glusa.
\newblock Towards an efficient finite element method for the integral
  fractional {L}aplacian on polygonal domains.
\newblock In {\em Contemporary Computational Mathematics - A Celebration of the
  80th Birthday of Ian Sloan}, pages 17--57. Springer International Publishing,
  2018.

\bibitem{BCH2020}
R.~Bailo, J.~A. Carrillo, and J.~Hu.
\newblock Fully discrete positivity-preserving and energy-dissipating schemes
  for aggregation-diffusion equations with a gradient-flow structure.
\newblock {\em Commun. Math. Sci.}, 18(5):1259--1303, Sept. 2020.

\bibitem{BCM2020}
R.~Bailo, J.~A. Carrillo, H.~Murakawa, and M.~Schmidtchen.
\newblock Convergence of a fully discrete and energy-dissipating finite-volume
  scheme for aggregation-diffusion equations.
\newblock {\em Math. Models Methods Appl. Sci.}, 30(13):2487--2522, Nov. 2020.

\bibitem{BARTIER201176}
J.-P. Bartier, A.~Blanchet, J.~Dolbeault, and M.~Escobedo.
\newblock Improved intermediate asymptotics for the heat equation.
\newblock {\em Appl. Math. Lett.}, 24(1):76--81, 2011.

\bibitem{biler2003generalised}
P.~Biler and G.~Karch.
\newblock Generalized fokker-planck equations and convergence to their
  equilibria.
\newblock {\em Banach Center Publ.}, 60:307--318, 2003.

\bibitem{blumenthal1960TheoremsStableProcesses}
R.~M. Blumenthal and R.~K. Getoor.
\newblock Some theorems on stable processes.
\newblock {\em Trans. Amer. Math. Soc.}, 95(2):263--273, 1960.

\bibitem{BBN2018}
A.~Bonito, J.~P. Borthagaray, R.~H. Nochetto, E.~Ot{\'{a}}rola, and A.~J.
  Salgado.
\newblock Numerical methods for fractional diffusion.
\newblock {\em Comput. Vis. Sci.}, 19(5-6):19--46, mar 2018.

\bibitem{BC10}
N.~Bournaveas and V.~Calvez.
\newblock The one-dimensional {K}eller-{S}egel model with fractional diffusion
  of cells.
\newblock {\em Nonlinearity}, 23(4):923--935, 2010.

\bibitem{Brezis2010}
H.~Brezis.
\newblock {\em Functional Analysis, Sobolev Spaces and Partial Differential
  Equations}.
\newblock {Springer}, {New York}, 2010.

\bibitem{Bucur2016}
C.~Bucur.
\newblock Some observations on the green function for the ball in the
  fractional {L}aplace framework.
\newblock {\em Commun. Pure Appl. Anal.}, 15(2):657--699, jan 2016.

\bibitem{carrillo2007strict}
J.~Carrillo, M.~Di~Francesco, and G.~Toscani.
\newblock Strict contractivity of the 2-wasserstein distance for the porous
  medium equation by mass-centering.
\newblock {\em Proc. Amer. Math. Soc.}, 135(2):353--363, 2007.

\bibitem{CCH2015}
J.~A. Carrillo, A.~Chertock, and Y.~Huang.
\newblock A finite-volume method for nonlinear nonlocal equations with a
  gradient flow structure.
\newblock {\em Commun. Comput. Phys.}, 17(1):233--258, 2015.

\bibitem{CFS20}
J.~A. Carrillo, F.~Filbet, and M.~Schmidtchen.
\newblock Convergence of a finite volume scheme for a system of interacting
  species with cross-diffusion.
\newblock {\em Numer. Math.}, 145(3):473--511, 2020.

\bibitem{cayama2021pseudospectral}
J.~Cayama, C.~M. Cuesta, and F.~de~la Hoz.
\newblock A pseudospectral method for the one-dimensional fractional laplacian
  on {$\mathbb{R}$}.
\newblock {\em Appl. Math. Comput.}, 389:125577, 2021.

\bibitem{CdTGP18}
N.~Cusimano, F.~del Teso, L.~Gerardo-Giorda, and G.~Pagnini.
\newblock Discretizations of the spectral fractional {L}aplacian on general
  domains with {D}irichlet, {N}eumann, and {R}obin boundary conditions.
\newblock {\em SIAM J. Numer. Anal.}, 56(3):1243--1272, 2018.

\bibitem{de2022stationary}
N.~De~Nitti and S.~Sakaguchi.
\newblock The stationary critical points of the fractional heat flow.
\newblock {\em Preprint arXiv: 2212.05383}, 2022.

\bibitem{dT14}
F.~del Teso.
\newblock Finite difference method for a fractional porous medium equation.
\newblock {\em Calcolo}, 51(4):615--638, 2014.

\bibitem{DiNezzaPalatucciValdinoci2012}
E.~Di~Nezza, G.~Palatucci, and E.~Valdinoci.
\newblock Hitchhiker's guide to the fractional {{Sobolev}} spaces.
\newblock {\em Bull. des Sci. Math.}, 136(5):521--573, 2012.

\bibitem{E06}
C.~Escudero.
\newblock The fractional {K}eller-{S}egel model.
\newblock {\em Nonlinearity}, 19(12):2909--2918, 2006.

\bibitem{EK2000}
R.~Estrada and R.~P. Kanwal.
\newblock {\em Singular Integral Equations}.
\newblock Birkhäuser Boston, 2000.

\bibitem{gentil2008levy}
I.~Gentil and C.~Imbert.
\newblock The {L}{\'e}vy--{F}okker--{P}lanck equation: $\phi$-entropies and
  convergence to equilibrium.
\newblock {\em Asymptot. Anal.}, 59(3-4):125--138, 2008.

\bibitem{HuangOberman2014}
Y.~Huang and A.~Oberman.
\newblock Numerical methods for the fractional {{Laplacian}}: {{A}} finite
  difference-quadrature approach.
\newblock {\em SIAM J. Numer. Anal.}, 52(6):3056--3084, 2014.

\bibitem{Ito_1998}
K.~Ito and F.~Kappel.
\newblock The trotter-kato theorem and approximation of pdes.
\newblock {\em Math. Comp.}, 67(221):21--44, 1998.

\bibitem{kato1978trotter}
T.~Kato.
\newblock Trotter's product formula for an arbitrary pair of self-adjoint
  contraction semigroup.
\newblock {\em Topics in Func. Anal., Adv. Math. Suppl. Studies}, 3:185--195,
  1978.

\bibitem{Kwasnicki2017}
M.~Kwasnicki.
\newblock Ten equivalent definitions of the fractional laplace operator.
\newblock {\em Fract. Calc. Appl. Anal.}, 20(1):7--51, 2017.

\bibitem{LS19}
L.~Lafleche and S.~Salem.
\newblock Fractional {K}eller-{S}egel equation: global well-posedness and
  finite time blow-up.
\newblock {\em Commun. Math. Sci.}, 17(8):2055--2087, 2019.

\bibitem{LeVeque1990}
R.~J. LeVeque.
\newblock {\em Numerical Methods for Conservation Laws}.
\newblock Birkhäuser Basel, 1990.

\bibitem{LeVeque2002}
R.~J. LeVeque.
\newblock {\em Finite Volume Methods for Hyperbolic Problems}.
\newblock Cambridge Texts in Applied Mathematics. Cambridge University Press,
  2002.

\bibitem{LR09}
D.~Li and J.~Rodrigo.
\newblock Finite-time singularities of an aggregation equation in {$\mathbb
  R^n$} with fractional dissipation.
\newblock {\em Comm. Math. Phys.}, 287(2):687--703, 2009.

\bibitem{LPG2020}
A.~Lischke, G.~Pang, M.~Gulian, F.~Song, C.~Glusa, X.~Zheng, Z.~Mao, W.~Cai,
  M.~M. Meerschaert, M.~Ainsworth, and G.~E. Karniadakis.
\newblock What is the fractional {L}aplacian? a comparative review with new
  results.
\newblock {\em J. Comput. Phys.}, 404:109009, mar 2020.

\bibitem{mao2017hermite}
Z.~Mao and J.~Shen.
\newblock Hermite spectral methods for fractional {PDEs} in unbounded domains.
\newblock {\em SIAM J. Sci. Comput.}, 39(5):A1928--A1950, 2017.

\bibitem{NOS15}
R.~H. Nochetto, E.~Ot\'{a}rola, and A.~J. Salgado.
\newblock A {PDE} approach to fractional diffusion in general domains: a priori
  error analysis.
\newblock {\em Found. Comput. Math.}, 15(3):733--791, 2015.

\bibitem{Saad2003}
Y.~Saad.
\newblock {\em Iterative Methods for Sparse Linear Systems}.
\newblock Society for Industrial and Applied Mathematics, Jan. 2003.

\bibitem{sheng2020fast}
C.~Sheng, J.~Shen, T.~Tang, L.-L. Wang, and H.~Yuan.
\newblock Fast {F}ourier-like mapped {C}hebyshev spectral-{G}alerkin methods
  for {PDEs} with integral fractional {L}aplacian in unbounded domains.
\newblock {\em SIAM J. Numer. Anal.}, 58(5):2435--2464, 2020.

\bibitem{Stein1970}
E.~M. Stein.
\newblock {\em Singular Integrals and Differentiability Properties of
  Functions}.
\newblock {Princeton University Press}, 1970.

\bibitem{trotter1959product}
H.~F. Trotter.
\newblock On the product of semi-groups of operators.
\newblock {\em Proc. Amer. Math. Soc.}, 10(4):545--551, 1959.

\bibitem{Valdinoci2009}
E.~Valdinoci.
\newblock From the long jump random walk to the fractional {{Laplacian}}.
\newblock {\em SeMA J. Bolet\'in la Soc. Espa\~nola Matem\'atica Apl.},
  49:1--7, 2009.

\bibitem{xu2021asymptotic}
W.~Xu and L.~Wang.
\newblock An asymptotic preserving scheme for {L\'evy-Fokker-Planck} equation
  with fractional diffusion limit.
\newblock {\em Preprint arXiv: 2103.08848}, 2021.

\end{thebibliography}
 
\end{document}